\documentclass{article}
\usepackage{graphicx}
\usepackage{amsmath}
\usepackage{amsfonts}
\usepackage{amssymb}
\usepackage{subfig}
\usepackage{float}
\usepackage{amsthm}
\usepackage{tikz}
\usepackage{tabularx}
\usepackage{appendix}
\usepackage{authblk}
\usepackage[backend=bibtex]{biblatex}
\usepackage{enumitem}
\numberwithin{equation}{section}
\usepackage[margin=1in]{geometry}
\newtheorem{theorem}{Theorem}[section]

\newtheorem{lemma}[theorem]{Lemma}
\newtheorem{proposition}[theorem]{Proposition}

\theoremstyle{definition}
\newtheorem{definition}{Definition}[section]
\newtheorem{remark}{Remark}[theorem]

\newtheorem{example}{Example}[section]

\usepackage{todonotes}

\let\amssymbboxplus\boxplus

\renewcommand{\boxplus}{\mathbin{\mathop\amssymbboxplus}}

\newcommand{\re}{\mathrm{Re}}
\newcommand{\im}{\mathrm{Im}}
\newcommand{\R}{\mathbb{R}}

\newcommand{\C}{\mathbb{C}}
\newcommand{\dd}{\,\mathrm{d}}
\newcommand{\D}{\mathbb{D}}
\newcommand{\Nat}{\mathbb{N}}
\newcommand{\set}[1]{\left\{#1\right\}}
\newcommand{\ie}{\textit{i.e.}}
\newcommand{\eg}{\textit{e.g.}}

\newcommand{\connectedcountersymbol}{C}
\newcommand{\graphname}{modified graph}
\newcommand{\graphsymbol}{\mathcal{G}}
\newcommand{\ind}{\mathbf{1}}
\newcommand{\I}{\mathrm{i}}

\def\Xint#1{\mathchoice
   {\XXint\displaystyle\textstyle{#1}}%
   {\XXint\textstyle\scriptstyle{#1 }}%
   {\XXint\scriptstyle\scriptscriptstyle{#1}}%
   {\XXint\scriptscriptstyle\scriptscriptstyle{#1}}%
   \!\int}
\def\XXint#1#2#3{{\setbox0=\hbox{$#1{#2#3}{\int}$}
     \vcenter{\hbox{$#2#3$}}\kern-.5\wd0}}

\def\dashint{\Xint-}

\title{Computing Inverses of Stieltjes Transforms of Probability Measures}
\author[1]{James Chen}
\author[2]{Sheehan Olver}
\affil[1]{Department of Applied Mathematics and Theoretical Physics, University of Cambridge}
\affil[2]{Department of Mathematics, Imperial College London}

\begin{document}

\maketitle

\begin{abstract}

The Stieltjes (or sometimes called the Cauchy) transform is a fundamental object associated with probability measures, corresponding to the generating function of the moments. In certain applications such as free probability it is essential to compute the inverses of the Stieltjes transform, which might be multivalued. This paper establishes conditions bounding the number of inverses based on properties of the measure which can be combined with contour integral-based root finding algorithms to rigorously compute all inverses. 

\end{abstract}

\section{Introduction}
\label{section:introduction}
Given a Borel probability measure $\mu$ on $\R$ with support $\Gamma$, we define its Stieltjes transform as:
\begin{equation}
    \label{eq:cauchytransform}
    G_{\mu}(z) = \int_{\Gamma} \frac 1{z-x} \dd \mu(x)
\end{equation}
for all $z \in \C \setminus \Gamma$. We also define its Hilbert transform as 
\begin{equation}
    \label{eq:hilberttransform}
    H_{\mu}(x) = \dashint_{\Gamma} \frac 1{x-t} \dd \mu(t)
\end{equation}
for all $x \in \R$.
The Stieltjes transform is a fundamental tool in probability, spectral theory, integrable systems, orthogonal polynomials, and elsewhere, as it is an analytic function that encodes the generating function of the moments of the underlying measure and characterises the measure uniquely.

In the analytic approach to free probability theory  \cite{VOICULESCU1986323,voiculescu1992free} one is interested in the functional inverse of the Stieltjes transform as they are used to define the free convolution of two probability measures, denoted by $\mu_a \boxplus \mu_b$. The output of the convolution is a probability measure that describes the distribution of the sum of a pair $a$ and $b$ of freely independent, self-adjoint noncommutative random variables. This convolution typically arises as the eigenvalue distribution of the sum of two Hermitian random matrices $A_n$ and $U_n B_n U_n^*$ as the size of the matrix $n$ tends to infinity, where $A_n$ and $B_n$ are random matrices and $U_n$ is a Haar distributed unitary matrix~\cite{mingospeicher2017free}.
The key idea is to use the $R$-transform, which is a complex analytic function satisfying the relations
\begin{equation*}
    G_{\mu}\left(R_{\mu}(z) +\frac 1z \right) = z.
\end{equation*}
The $R$-transform then satisfies the relation
\begin{equation}
    \label{eq:rtransformconv}
    R_{\mu_a \boxplus \mu_b}(z) = R_{\mu_a}(z) + R_{\mu_b}(z)
\end{equation}
where $\mu_a$ and $\mu_b$ are the distributions of the freely independent random variables $a,b$. The $R$-transform provides an avenue for computing free convolutions either symbolically \cite{rao2008polynomial} or via the computation of inverses of the Stieltjes transform.

In~\cite{olver2013numerical} a numerical approach to compute free convolutions was introduced built on top of computing inverses of Stieltjes transforms.  The compact support of the measures combined with the invertibility of the Stieltjes transform near infinity (Lemma~\ref{lemma:cauchyproperties}) allows \eqref{eq:rtransformconv} to be rewritten as
\begin{equation*}
    G_{\mu_a \boxplus \mu_b}^{-1}(\zeta) = G_{\mu_a}^{-1}(\zeta) + G_{\mu_b}^{-1}(\zeta) - \frac 1\zeta
\end{equation*}
which holds in a neighbourhood of the origin. The inverse of the Stieltjes transform is then computed for square-root decaying measures and pure point using polynomial root finding (See Appendix~\ref{appendix:pp} and~\ref{appendix:sq} for a discussion) and for general measures supported on a compact interval using the FFT. While the computation succeeds for computing some of the inverses for measures which are not univalent, it  will fail to compute inverses lying between components of the support of the measure: the algorithm works by re-expanding in an ellipse surrounding the support of the measure and will not detect roots within the ellipse. Moreover, there is no guarantee that all roots are calculated. 
In very recent work, Cortinovis and Ying \cite{cortinovis2023computingfreeconvolutionscontour} compute the $R$-transform  directly using contour integration, which provides a method to compute the free convolution of measures with support on compact intervals with univalent Stieltjes transforms. This approach require that both measures involved have univalent Stieltjes transform on the whole domain, which excludes the case of measures supported on multiple intervals. The major complication is that measures with disconnected support are in general \textit{not} univalent.

In this paper, we provide rigorous upper bounds on how many inverses the Stieltjes transform of a compactly supported measure may have. This allows for the development of a method to recover \textit{all} inverses of a Stieltjes transform of a compactly supported measure under broad conditions, including measures supported on multiple intervals of support and measures with jump discontinuities at the boundary of their support.

The paper is organised as follows. In Section~\ref{section:preliminaries} we recall the properties of the Stieltjes transform relevant to our discussion. Section~\ref{section:acmeasures} will prove bounds on how many inverses the Stieltjes transform of compactly supported measures with sufficient smoothness in their densities. Such bounds will be extended to a wider class of measures in Section~\ref{section:generalmeasures}, including measures with discontinuities in their densities. We address the convergence of computed inverses when the density of the measure $\rho(x)$ is approximated by a sequence $\rho_n(x)$. Lastly, in Section~\ref{section:numericalexperiments} we conduct numerical experiments with various measures, including computing the free convolution of measures which do not have connected support, but nevertheless have univalent Stieltjes transform and thus fall within the current theory.

\begin{remark}
In Free Probability what we call the Stieltjes transform is typically called the Cauchy transform, whereas in complex analysis the Cauchy transform usually has the normalisation $\frac{1}{(-2\pi\I)}G_\mu(z)$.  In some contexts the Stieltjes transform is defined with the opposite sign as $m_{\mu}(z) = -G_{\mu}(z)$. In this case, the Stieltjes transform maps the upper-half complex plane $\C^+$ to itself. These other normalisations result in a slightly different formulation of the Sokhotski--Plemelj theorem stated below.
\end{remark}

{\it Acknowledgements}: We thank Raj Rao Nadakuditi and Thomas Trogdon for helpful guidance and suggested references used in preparing this paper.

\section{Preliminaries}
\label{section:preliminaries}

Some well-known properties of the Stieltjes transform are listed in the following lemma below.
\begin{lemma}\cite[Lemmas 2 and 3, Chapter 3]{mingospeicher2017free}
    \label{lemma:cauchyproperties}Let $\mu$ be a Borel probability measure on $\R$ with support $\Gamma$. Then
    \begin{itemize}
        \item $G_{\mu}$ is an analytic function on the set $\C \setminus \Gamma$.
        \item $G_{\mu}$ commutes with complex conjugation {\ie}  $\overline{G_{\mu}(z)} = G_{\mu}(\overline{z})$
        \item if $\im(z) \geq 0$, then $-\im(G(z)) \geq \im(z)$. In particular, $G$ maps the upper-half of the complex plane $\C^+$ to the lower-half $\C^-$ and vice versa.
        \item $\lim_{y \to \infty} \I y G_{\mu}(\I y) = 1$. In particular, if the support of $\mu$ is bounded, then $G_{\mu}(z) = \frac 1z + O(\frac 1{z^2})$ in a neighbourhood of infinity.
    \end{itemize}
\end{lemma}

We will also make use of the Sokhotski--Plemelj theorem, which links together the density, its Stieltjes transform and its Hilbert transform.
\begin{theorem}[Sokhotski--Plemelj theorem~\cite{muskhelishvili1977singular}]
    \label{theorem:sokhotskiplemeljtheorem}
    Suppose $\mu$ is a probability measure that admits a density $\rho(x)$ which is H{\"o}lder continuous. Denote $G_{\mu}^\pm (x) = \lim _{\varepsilon \downarrow 0} G_{\mu}(x \pm i \varepsilon)$ for $x \in \R$. Then we have
    \begin{align}
        G_{\mu}^+(x) + G_{\mu}^-(x) &= 2 \dashint_{\R} \frac{\rho(t)}{x-t} \dd t = 2 H_{\mu}(x)\nonumber \\
        G_{\mu}^+(x) - G_{\mu}^-(x) &= -2\pi \I \rho(x) \nonumber
    \end{align}
\end{theorem}
Note that in the above setting, the limit $\lim _{z_n \to 0}G_{\mu}(x + z_n)$ also exists for $x \in \R$ and $z_n$ in $\C^+$, provided that the limit is \textit{non-tangential} {\ie} there exists $\varepsilon>0$ such that $\arg(z_n) \in (\varepsilon, \pi - \varepsilon)$ for all $n$. A similar statement holds for $z_n \in \C^-$.

In the following sections, we will make use of the Joukowski transform, which maps the open unit disc $\D$ to the slit complex plane $\C \cup \set{\infty} \setminus [-1,1]$, with $0$ being mapped to the point at infinity.
    \begin{equation*}
        J(z) = \frac12 \left(z + \frac 1z\right)
    \end{equation*}
We also define an affine function which sends the interval $[-1,1]$ to $[a,b]$.
    \begin{equation*}
        M_{(a,b)}(z) = \frac{b+a}{2} + \frac{b-a}{2}z
    \end{equation*}

Lastly, we will need to count the number of connected components of subsets of $\R$ in order to state the main theorems proved in this paper.

\begin{definition}[Number of connected components]
    Given a set $A \subset \R$, we define an equivalence relation stating that for points $a,b \in A$, $a \sim b$ if the closed interval $[a,b]$ is a subset of $A$. The connected components are the equivalence classes under this relation. We can then define $\connectedcountersymbol:\mathcal{P}(\R) \to \Nat \cup \set{\infty}$ as
    \begin{equation*}
        \connectedcountersymbol(A) = \text{The number of connected components in the set } A
    \end{equation*}
    where we define that $C(A) = \infty$ whether the number of connected components is countably or uncountably infinite. Here, $\mathcal{P}(\R)$ denotes the set of all subsets of $\R$.
\end{definition}

\section{Bounds on inverses of Stieltjes transforms for measures with absolutely continuous density}
\label{section:acmeasures}

In this section, we will consider a non-zero absolutely continuous measure $\dd\mu(x) = \rho(x) \dd x$ with compact support, which we denote with $\Gamma$. The density $\rho(x)$ is assumed to be an absolutely continuous representative of the Sobolev space $W^{1,p}(\R)$ for some $1<p<\infty$. It is straightforward to verify that the above condition implies that $\rho(x)$ is H{\"o}lder continuous for exponent $\alpha = 1-\frac{1}{p}$ on the whole real line due to the Sobolev embedding theorem.

Throughout this section, we will count roots of holomorphic functions up to multiplicity. We will also assume that we always take the absolutely continuous representative whenever elements of the Sobolev space $W^{1,p}$ are mentioned. Lastly, we will define the convex hull of the support $\Gamma$ as the interval $[a,b]$ where $a = \min(\Gamma)$, $b = \max(\Gamma)$.

We begin by making precise the notion that if a curve has winding number $k$ around a point $z_0$, then it must cross a ray emanating from $z_0$ at least $k$ times as well.
\begin{lemma}
    \label{lemma:raybounds}
    Let $\gamma:[0,1] \to \C \setminus \{z_0\}$ be a continuous closed curve. Then define the ray centered at $z_0 \in \C$ with angle $\theta \in [0, 2\pi)$ as the set
    \begin{equation*}
        R_{z_0, \theta} = \set{z_0 + re^{\I\theta} : r \geq 0}
    \end{equation*}
    In addition, assume $\gamma(0) \notin R_{z_0, \theta}$.
    Then for any $\theta$, the closed set
    \begin{equation*}
        A = \set{t \in [0,1] : \gamma(t) \in R_{z_0, \theta}}
    \end{equation*}
    consists of at least $\left |\mathrm{Ind}_{\gamma}(z_0)\right|$ disjoint (possibly degenerate) intervals, where $\mathrm{Ind}_{\gamma}(z_0)$ denotes the winding number of $\gamma$ around $z_0$.
    In other words,
    \begin{equation*}
        \connectedcountersymbol(A) \geq \left|\mathrm{Ind}_{\gamma}(z_0)\right|
    \end{equation*}
\end{lemma}

\begin{proof}
    Let $\gamma:[0,1] \to \C \setminus \{z_0\}$ be a continuous closed curve.
    Then there exists continuous real-valued functions $r_{\gamma}$ and $\theta_{\gamma}$ such that 
    \begin{equation*}
        \gamma(t) = z_0 + r_{\gamma}(t)e^{\I \theta_{\gamma}(t)}
    \end{equation*}
    where $r_{\gamma} >0$ (Theorem 7.2.1 in~\cite{beardon1979complex}). The curve $\gamma$ crosses the ray $R_{z_0, \theta}$ whenever $\theta_{\gamma}(t) \equiv \theta \mod 2\pi$. By the intermediate value theorem, this must happen at least $\left|\frac{\theta_{\gamma}(1) - \theta_{\gamma}(0)}{2\pi} \right|$ times, which is the absolute value of the winding number $k = \mathrm{Ind}_{\gamma}(z_0)$. Therefore, $A$ contains at least $k$ closed (possibly degenerate) intervals, which must be disjoint.
\end{proof}

\begin{definition}[Image of the support]
    Let $\mu$ be a compactly supported Borel measure that admits a H{\"o}lder continuous density $\rho(x)$. We define continuous closed curves $\gamma_r : [0,2\pi] \to \C$ for each $r \in (0,1]$ by
    \begin{align}
        \label{eq:pointwiselimit}
        \gamma_r(t) &= G_{\mu}(M_{(a,b)}(J(re^{\I t})))\hbox{ for $r \neq 1$}, \nonumber\\
        \gamma_1(t) &= \lim_{r \uparrow 1} \gamma_r(t),
    \end{align}
where $a$ and $b$ are such that the convex hull of the support $\Gamma$ is $[a,b]$.
\end{definition}
Geometrically, $\gamma_1$ is the image of the Stieltjes transform with the limiting values above and below the support, traversing around the support in a clockwise fashion. By the Sokhotski--Plemelj theorem (Theorem~\ref{theorem:sokhotskiplemeljtheorem}) and the H{\"o}lder continuity of $\rho(x)$, the pointwise limit in \eqref{eq:pointwiselimit} exists for each $t$ and is given by

\begin{equation}
    \label{eq:parametrisation}
    \gamma_1(t) = 
    \left\{\begin{matrix}
        H_{\mu}(M_{(a,b)}(\cos(t))) + \pi \I \hspace{0.1cm} \rho(M_{(a,b)}(\cos(t))) & t \in [0, \pi] \\
        H_{\mu}(M_{(a,b)}(\cos(t))) - \pi \I \hspace{0.1cm} \rho(M_{(a,b)}(\cos(t))) & t \in (\pi,2\pi]\\
       \end{matrix}\right.
\end{equation}
where $H_{\mu}$ is the Hilbert transform of $\mu$ \eqref{eq:hilberttransform}. It is well known that the Hilbert transform preserves the H{\"o}lder exponent of the density (and is even bounded on the space $C^{0,\alpha}([a,b])$ for $0<\alpha<1$, see Theorem 4.7 in~\cite{mikhlin1987singular}). Thus, $\gamma_1$ forms a continuous closed curve. We now show that the above limit in \eqref{eq:pointwiselimit} holds uniformly.
\begin{lemma}
    Let $\mu$ be a compactly supported Borel measure that admits a H{\"o}lder continuous density $\rho(x)$. Then,
    \label{lemma:uniformpathconv}
    \begin{equation*}
        \lim_{r \uparrow 1} \sup_{t \in [0,2\pi]} | \gamma_r(t)  - \gamma_1(t)| = 0
    \end{equation*}
\end{lemma}
\begin{proof}
    
    First consider $t \in [\pi, 2\pi]$. Then $M_{(a,b)}(J(re^{\I t}))$ only takes values in $\im(z) \geq 0$ and in this upper-half plane, $G_{\mu}$ admits a continuous extension to the real line, due to Plemelj's theorem and the H{\"o}lder continuity of the density. Then $G_{\mu}$ is uniformly continuous on a sufficiently large compact rectangle in the upper-half plane, e.g., 
    \begin{equation*}
        D = \set{z \in \C : \re(z) \in [a-1, b+1], \im(z) \in [0,1]}
    \end{equation*}
    and the uniform convergence of $J(re^{\I t})$ to $\cos(t)$ gives the desired result on $[\pi,2\pi]$. A similar argument can be done for $t \in [0, \pi]$ in the lower-half plane, which results in uniform convergence on $[0,2\pi]$.

\end{proof}

We begin the proof of Theorem~\ref{theorem:bounds1} by first proving a bound for the case where $\zeta$ does not lie on the image of $\gamma_1$.
\begin{proposition}
    \label{proposition:boundoffcurve}
    Let $\mu$ be a compactly supported Borel measure that admits a H{\"o}lder continuous density $\rho(x)$. Let $\zeta \in \C \setminus \R$ such that $\zeta$ does not lie in the image of $\gamma_1(t)$.
    Then the number of solutions to $G_{\mu}(z) = \zeta$ is bounded above by
    \begin{equation*}
        \frac 12 \connectedcountersymbol\left(\set{x \in \R: \rho(x) = \frac{|\im(\zeta)|}{\pi}}\right),
    \end{equation*}
\end{proposition}
\begin{proof}
    By Lemma~\ref{lemma:uniformpathconv}, there exists $r'$ such that for any $r \geq r'$ we have $\sup_{t \in [0,2\pi]} | \gamma_r(t)  - \gamma_1(t)| < d(\zeta, \gamma_1)$. It follows that the curves $\gamma_1$ and $\gamma_{r}$ are homotopic for these values of $r$ and hence their winding numbers are equal, \ie,
    \begin{equation*}
        \mathrm{Ind}_{\gamma_r}(\zeta) = \mathrm{Ind}_{\gamma_1}(\zeta).
    \end{equation*}
    By the argument principle, the index $\mathrm{Ind}_{\gamma_r}(\zeta)$ is the number of solutions to $G_{\mu}(z) = \zeta$ which lie outside of the ellipse whose boundary is given by $M_{(a,b)}(J(re^{\I t}))$.
    Define the rays $R_{\zeta, 0}$ and $R_{\zeta, \pi}$, whose union is a line passing through $\zeta$ parallel to the real line. By Lemma~\ref{lemma:raybounds} and \eqref{eq:parametrisation}, we have
    \begin{align}
        \label{eq:propconnectedA}
        \connectedcountersymbol\left(\set{x \in \R: \rho(x) = \frac{|\im(\zeta)|}{\pi}}\right)
        &= \connectedcountersymbol\left(\set{t \in [0,2\pi] : \im(\gamma_1(t)) = \im(\zeta)}\right)\\
        &= \connectedcountersymbol(\set{t \in [0,2\pi] : \gamma_1(t) \in R_{\zeta, 0}}) +  \connectedcountersymbol(\set{t \in [0,2\pi] : \gamma_1(t) \in R_{\zeta, \pi}}) \nonumber \\ 
        &\geq 2 |\mathrm{Ind}_{\gamma_1}(\zeta)|\nonumber 
    \end{align}

    The first equality holds because the imaginary part of $\gamma_1$ is essentially $\rho_{[a,b]}(-x)$ and $-\rho_{[a,b]}(x)$ concatenated together (up to a parametrisation), where $\rho_{[a,b]}$ is the density restricted to the convex hull $[a,b]$. Since the density $\rho(x)$ is positive, we have that depending on whether $\im(\zeta)$ is positive or negative, $\im(\gamma_1(t)) = \im(\zeta)$ can only occur when $t \in (0,\pi)$ or $(\pi, 2\pi)$ respectively.

    The second equality holds as the two sets form a partition of \eqref{eq:propconnectedA} with each connected component of \eqref{eq:propconnectedA} being present in exactly one of the sets. Taking the limit as $r \to 1$ yields the result, since no solutions of $G_{\mu}(z) = \zeta$ lie on the real line.
\end{proof}
\begin{remark}
    With regards to the first equality in \eqref{eq:propconnectedA}, the above argument can be modified to bound the number of inverses of the Stieltjes transform of a H{\"o}lder continuous \textit{function} which may not be non-negative \ie
    \begin{equation*}
        G_f(z) = \int_{\Gamma} \frac{f(x)}{z-x}\dd x
    \end{equation*}
    Then, the number of inverses is bounded above by 
    \begin{equation*}
        \frac 12 \connectedcountersymbol\left(\set{x \in \R: |f(x)| = \frac{|\im(\zeta)|}{\pi}}\right)
    \end{equation*}
\end{remark}
We now wish to extend Proposition~\ref{proposition:boundoffcurve} to the whole complex plane. However, the curve $\gamma_1$ may be pathological for a general measure with a H{\"o}lder continuous density. In particular, we cannot exclude the case that the image of $\gamma_1$ has non-empty interior in some region of $\C$ with just the assumption of H{\"o}lder continuity. However, if we assume that the density $\rho(x)$ is in $W^{1,p}(\R)$ for some $1<p<\infty$, Proposition~\ref{proposition:rectifiablelimit} will show that $\gamma_1$ has finite length and thus has nowhere dense image in $\C$. Recall the definition of what it means for a curve to be \textit{rectifiable}.

\begin{definition}[Rectifiable curve]
    Let $\gamma:[a,b] \to \C $ be a curve. Define the arclength of $\gamma$ to be the quantity
    \begin{equation*}
        |\gamma| = \sup_{P} \sum_{i=1}^{n} |\gamma(t_{i}) - \gamma(t_{i-1})|
    \end{equation*}
    where $P$ ranges over all partitions $a = t_0 < t_1 < \cdots < t_n = b$. The path $\gamma$ is said to be \textit{rectifiable} if $|\gamma|$ is finite.
\end{definition}

In preparation for this result, we first state a result about the Hilbert transform on Sobolev spaces.

\begin{lemma}[Theorem 6.2 in~\cite{mikhlin1987singular}, Chapter 2]
         \label{lemma:hilbertregularity}
    Let $I$ be a compact interval. Then given $f \in W^{1,p}(I)$ for some $1<p<\infty$, the Hilbert transform continuously maps into $W^{1,p}(J)$ for any interval $J$ whose closure lies entirely in the interior of $I$.
\end{lemma}

\begin{proposition}
    \label{proposition:rectifiablelimit}
    Let $\dd \mu = \rho(x) \dd x$ be a compactly supported measure with density in $W^{1,p}(\R)$ for some $1<p<\infty$. Then the curve $\gamma_1$ is rectifiable.
\end{proposition}

\begin{proof}
    It is sufficient to prove that the real and imaginary parts of $\gamma_1$ as real-valued functions of $t$ are of bounded variation on the interval $[a,b]$. From \eqref{eq:parametrisation}, we know that these are respectively the Hilbert transform of $\mu$ and the density of $\mu$ up to a multiplicative constant. By assumption $\rho(x)$ is absolutely continuous, and thus of bounded variation.
    Since $\rho(x) \in W^{1,p}(\R)$, the restriction to the interval $(a-1, b+1)$ is also in $W^{1,p}(a-1, b+1)$.
    By Lemma~\ref{lemma:hilbertregularity}, we have that $H_{\mu}(x) \in W^{1,p}(a,b)$ and thus is also absolutely continuous and of bounded variation.
\end{proof}

The other ingredient needed to extend Proposition~\ref{proposition:boundoffcurve} is the following lemma which allows for upper bounds on the number of inverses of a non-constant holomorphic function defined on a dense subset to be extended to the whole complex plane. Note that the Stieltjes transform of $\mu$ is constant if and only if $\mu$ is the zero measure, so we may exclude this case from consideration.
\begin{lemma}[Inverse bound on dense set]
    \label{lemma:densebound}
    Let $f:U \to \C$ be a non-constant holomorphic function on an open set $U$. Let $D$ be a dense subset of $\C$ and suppose that we know that for every $\zeta \in D$, the number of solutions (up to multiplicity) of $f(z) = \zeta$ is bounded above by $N \in \Nat$. Then there are at most $N$ solutions (up to multiplicity) of $f(z) = \zeta$ for all $\zeta \in \C$. 
\end{lemma}
\begin{proof}
    The statement is true for all $\zeta \in D$ by assumption. Let $\zeta \in \C \setminus D$ and suppose there are at least $M$ solutions to $f(z) = \zeta$. Note that the actual number of solutions could actually be countably infinite, but we only require a lower bound.
    Denote these distinct solutions $z_1, \dots, z_M$. Let $U_i$ be disjoint open neighbourhoods of $z_i$. Then by the open mapping theorem, $V = \bigcap_{i=1}^M G_{\mu}(U_i)$ is an open set containing $\zeta$.
    Choose $\zeta'$ in the non-empty set $D \cap V$, which must also have at least $M$ solutions to $f(z) = \zeta'$, since for each $i$, there must be $m_i$ solutions lying in $U_i$, where $m_i$ is the multiplicity of the solution $z_i$. We conclude that $M \leq N$ and the proof is complete.
\end{proof}

We are now in a position to prove upper bounds on the number of inverses for the whole complex plane.
\begin{theorem}[Bounds on inverses of Stieltjes transform 1]
    \label{theorem:bounds1}
    Let $\mu$ be a Borel measure supported on a compact set $\Gamma$ with a density $\rho(x) \in W^{1,p}(\R)$ for some $1<p<\infty$.
    
    Let $\zeta \in \C \setminus \R$.
    Then the number of solutions to $G_{\mu}(z) = \zeta$ is bounded above by
    \begin{equation*}
        N = \frac 12 \sup_{y > 0, y \notin S} \connectedcountersymbol\left(\set{x \in \R: \rho(x) = y}\right)
    \end{equation*}
    where $S$ is an arbitrary nowhere dense subset of $\R$.
\end{theorem}

\begin{proof}
    Since $\gamma_1$ is rectifiable by Proposition~\ref{proposition:rectifiablelimit}, the image of the curve is nowhere dense. Therefore, since $N$
    is an upper bound on the number of solutions to $G_{\mu}(z) = \zeta$ for all $\zeta$ not lying in the set $A = \gamma_1 \cup \R \cup \set{x+y\I : x \in \R, y \in S}$, by Proposition~\ref{proposition:boundoffcurve}, the bound holds for all $\zeta \in \C$ following an application of Lemma~\ref{lemma:densebound} to the nowhere dense set $A$.
\end{proof}

\begin{remark}
    \label{remark:dense}
    The weakening of the supremum to exclude a nowhere dense subset will be useful in excluding certain edge cases in Section~\ref{section:generalmeasures}, when we extend the results to measures with possible discontinuities in their density.
\end{remark}

Theorem~\ref{theorem:bounds1} tells us we can bound the number of inverses of a Stieltjes transform based on how ``wiggly" the density of the measure is. For instance,
\begin{itemize}
    \item Suppose $\rho(x)$ is a function which achieves its maximum value at $x_{\max}$ and is monotonic below and above this value. Then the Stieltjes transform is univalent.
    \item Suppose the measure $\mu$ has Jacobi behaviour {\ie}  $\rho(x) = p(x)(1-x)^{\alpha}(1+x)^{\beta}$, where $p(x)$ is a polynomial and $\alpha, \beta > 0$. Then the Stieltjes transform associated with this density has at most $\frac {\deg(p)+1}{2}$ inverses for any $\zeta \in \C$.
    \item More generally, if the density $\rho(x)$ is continuous and has $n$ turning points, then there are at most $\frac{n+1}{2}$ solutions. This method is most useful numerically as it can be estimated by sampling the density pointwise.
\end{itemize}
\begin{remark}
    \label{remark:horizontal}
    The above theorem can be interpreted as a \textit{horizontal line test} where the maximum number of intersection points between a horizontal line and the density is measured. This idea can be generalised to the case of measures which do not have continuous density, see Section~\ref{section:generalmeasures}.
\end{remark}
\begin{remark}
    One can also prove using the monotonicity of the Stieltjes transform that for $y \in \R$, the number of solutions to $G_{\mu}(z) = y$ is bounded above by the number of connected components in the support of $\mu$. It can also be shown that the bound given in Theorem~\ref{theorem:bounds1} is greater than or equal to the number of connected components in the support. However, Theorem~\ref{theorem:bounds1} applies to all of $\C$.
\end{remark}

\section{Bounds for more general measures}
\label{section:generalmeasures}

The goal of this section is to generalise Theorem~\ref{theorem:bounds1} to more general measures. In particular, we will consider measures which have densities which have jump discontinuities or have unbounded behaviour. The idea is to approximate measures with a sequence of measures which have known bounds. We need the following proposition.

\begin{proposition}[Theorem 3.1 of~\cite{Rao1962RelationsBW}, Theorem 8.2.18 of~\cite{bogachev2006measure}]
    \label{proposition:weakuniform}
    Let $\mu_n, \mu$ be Borel probability measures on $\R$ such that $\mu_n$ converges weakly to $\mu$. Let $\mathcal{A}$ be a family of uniformly bounded and equicontinuous functions. Then
    \begin{equation*}
        \lim_{n \to \infty} \sup_{f \in \mathcal{A}} \left| \int_{\R} f \dd \mu_n  - \int_{\R}  f \dd \mu \right| = 0
    \end{equation*}
\end{proposition}

Using this, we can prove the following useful proposition which allows us to bound the number of inverses of the Stieltjes transform for a particular Borel probability measure with a sequence of probability measures that converge weakly.
\begin{proposition}
    \label{proposition:weakbound}
    Let $\mu_n, \mu$ be finite Borel probability measures on $\R$ such that $\mu_n$ converges weakly to $\mu$. Suppose that there exists a constant $N$ such that $G_{\mu_n}(z) = \zeta$ has at most $N$ solutions for all $n \in \Nat$ and $\zeta \in \C$. Then $G_{\mu}(z) = \zeta$ has at most $N$ solutions.
\end{proposition}
\begin{proof}
    We show that $G_{\mu_n}$ converges uniformly on all compact subsets of $\C^+$ to $G_{\mu}$. Let $K \subset \C^{+}$ be compact and let $\Gamma$ be the support of $\mu$. Then we have
    \begin{equation*}
        \inf_{z \in K, x \in \R} |z-x| = \delta > 0
    \end{equation*}
    Define the family of functions $\mathcal{A} = \set{f_z:\R \to \C, f_z(x) = \frac{1}{z-x}: z \in K}$. It is clear that $\mathcal{A}$ is uniformly bounded above by $\frac{1}{\delta}$. The real and imaginary parts of each $f \in \mathcal{A}$ also have uniformly bounded derivative, hence $\mathcal{A}$ is (uniformly) Lipschitz and equicontinuous. Thus, by Proposition~\ref{proposition:weakuniform}, 
    \begin{equation*}
        \lim_{n \to \infty} \sup_{f \in \mathcal{A}} \left| \int_{\R} f \dd \mu_n  - \int_{\R} f \dd \mu \right| = \lim_{n \to \infty} \sup_{z \in K} \left| G_{\mu_n}(z)  - G_{\mu}(z) \right| = 0 
    \end{equation*}
    Hence, we have uniform convergence on compact sets. Thus, by Hurwitz's theorem, the number of solutions to $G_{\mu_n}(z) = \zeta$ and $G_{\mu}(z) = \zeta, z \in \C^{+}$ are the same for large enough $n$. By commuting with complex conjugation, the same result holds for the lower-half complex plane $\C^{-}$. We can extend the bound to the real line using Lemma~\ref{lemma:densebound}, hence the result.
\end{proof}
The above theorem can already be applied to measures that do not have a continuous density. We give a few examples below.

\begin{example}[Uniform Distribution]
    \label{example:uniformdist}
    Consider the Uniform Distribution $\dd \mu = \rho(x) \dd x = \frac 12 \ind_{[-1,1]} (x)\dd x$. We approximate $\rho(x)$ with Lipschitz densities $\rho_n(x)$, where $Z_n$ is a normalisation constant:
    \begin{equation*}
        \rho_n(x) = \frac 1{Z_n}\left\{\begin{matrix}
            n(x+1) & -1 \leq x < -1+\frac 1n \\
            1 & -1 + \frac 1n \leq x \leq 1 - \frac 1n\\
            n(1-x)& 1 - \frac 1n < x \leq 1 \\
           0 & \text{otherwise}
           \end{matrix}\right. .
    \end{equation*}
    The pointwise convergence of densities almost everywhere implies weak convergence in the corresponding measures.
    Applying Theorem~\ref{theorem:bounds1} yields a bound of $N = 1$ for each measure $\mu_n$ with density $\rho_n(x)$, hence each measure has univalent Stieltjes transform. Proposition ~\ref{proposition:weakuniform} tells us that $\mu$ has univalent Stieltjes transform as well. In this case, we already know that the uniform distribution has univalent Stieltjes transform, as one can calculate explicitly the inverse of the Stieltjes transform as $G^{-1}_{\mu}(\zeta) = \coth(\zeta)$.
\end{example}
\begin{example}[Point measures]
    \label{example:purepoint}
    Consider a pure point probability measure that is the sum of $N$ Dirac measures at isolated points, \ie,
    \begin{equation*}
        \mu = \sum_{i=1}^N \alpha_i \delta_{x_i}
    \end{equation*}
    where $x_i$ are pairwise distinct and $\sum_{i} \alpha_i = 1$. We may approximate each $\delta_{x_i}$ with measures with Lipschitz densities, \eg,
    \begin{align*}
        \dd \delta_{n,x_i} &= \rho_{n,x_i}(x) \dd x, \\
        \rho_{n,x_i}(x) &= n\max\left(1-n\left|x-x_i\right|,0\right).
    \end{align*}
    As $n$ tends to infinity, the measure $\delta_{n,x_i}$ converge weakly to $\delta_{x_i}$ and thus the measures $\mu_n$ converge weakly to $\mu$, where $\mu_n$ is defined by
    \begin{equation*}
        \mu_n = \sum_{i=1}^N \alpha_i \delta_{n, x_i}.
    \end{equation*}
    It is clear that Theorem~\ref{theorem:bounds1} yields a bound of $N$ when $n$ is large enough. Thus, any pure point measure with $N$ points of support can have at most $N$ inverses. In fact, such a pure point measure will always have $N$ inverses, see Appendix~\ref{appendix:pp} for further discussion.
\end{example}

In order to state Theorem~\ref{theorem:bounds2} and Theorem~\ref{theorem:bounds3}, let $U$ be a subset of $\R$ and $f:U \to \R$ a function with finitely many points of discontinuity whose left and right limits exist at every point $x \in U$, allowing for possibly infinite limits. We will define a modified version of the \textit{graph} of a function $f:U \to \R$ which we denote by the symbol $\graphsymbol(f)$.

    \begin{equation*}
        \mathcal{G}(f) = \set{(x,y) \in U \times \R : y \text{ lies (inclusively) between } \lim_{z \uparrow x} f(z) \text{ and } \lim_{z \downarrow x} f(z)}
    \end{equation*}
This set is always closed. In addition, if $f$ is bounded, then this set is path connected as well. Lastly, we note that if $f$ is continuous on $U$ then the graph and the {\graphname} are equal.

\begin{example}
    Consider the function
    \begin{equation*}
        f(x) = \left\{\begin{matrix}
           1 &  -3 \leq x \leq -2\\
           \frac{1}{\sqrt{1-x^2}} & -1<x<1 \\
           0 & \text{ otherwise} \\
           \end{matrix}\right..
    \end{equation*}
In Figure~\ref{fig:modifiedgraph} we plot both the graph and the \graphname. We note that the {\graphname} is path connected except at the points where the behaviour of $f$ is unbounded. This notion of graph allows us to generalise the notion of the \textit{horizontal line test} to measures whose density is not continuous, see Remark~\ref{remark:horizontal}.
\end{example}

\begin{figure}[H]
    \centering
    \subfloat[]{\includegraphics[width = 0.5\textwidth]{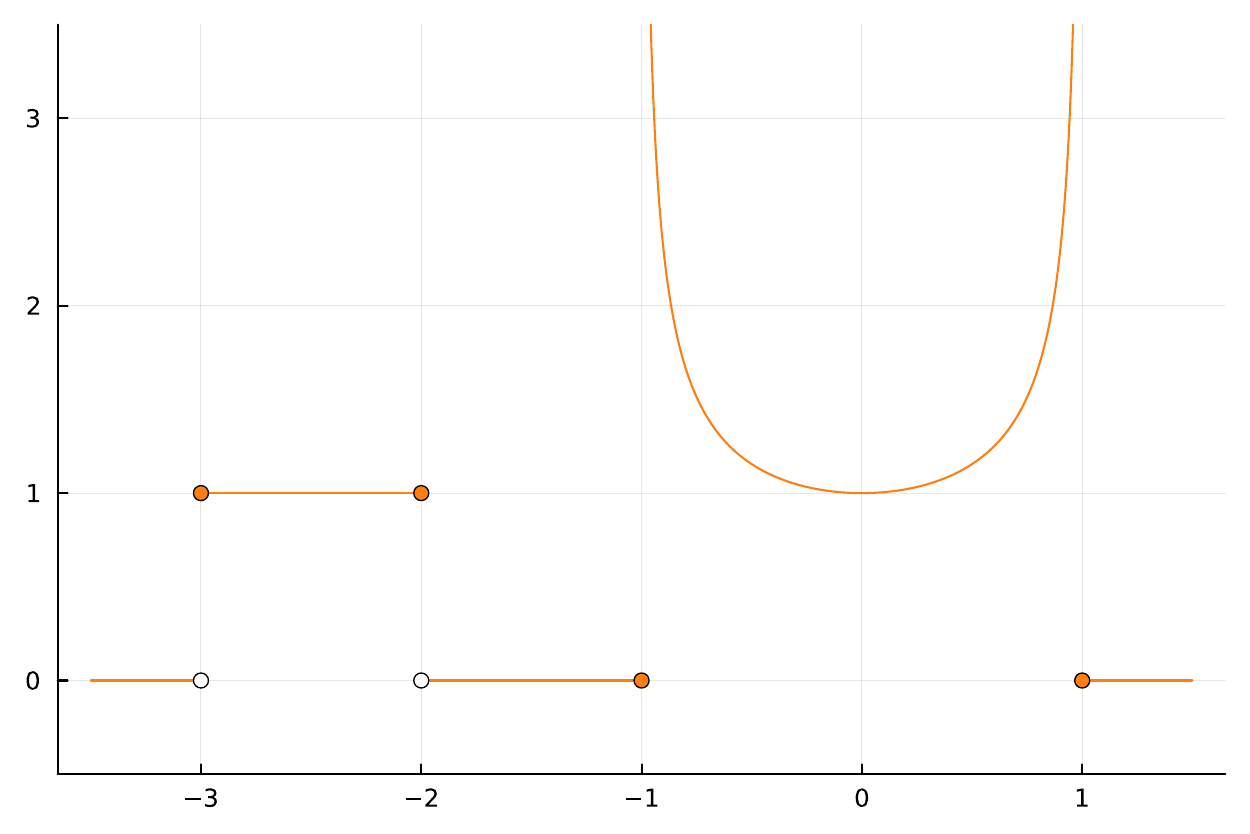}}
    \subfloat[]{\includegraphics[width = 0.5\textwidth]{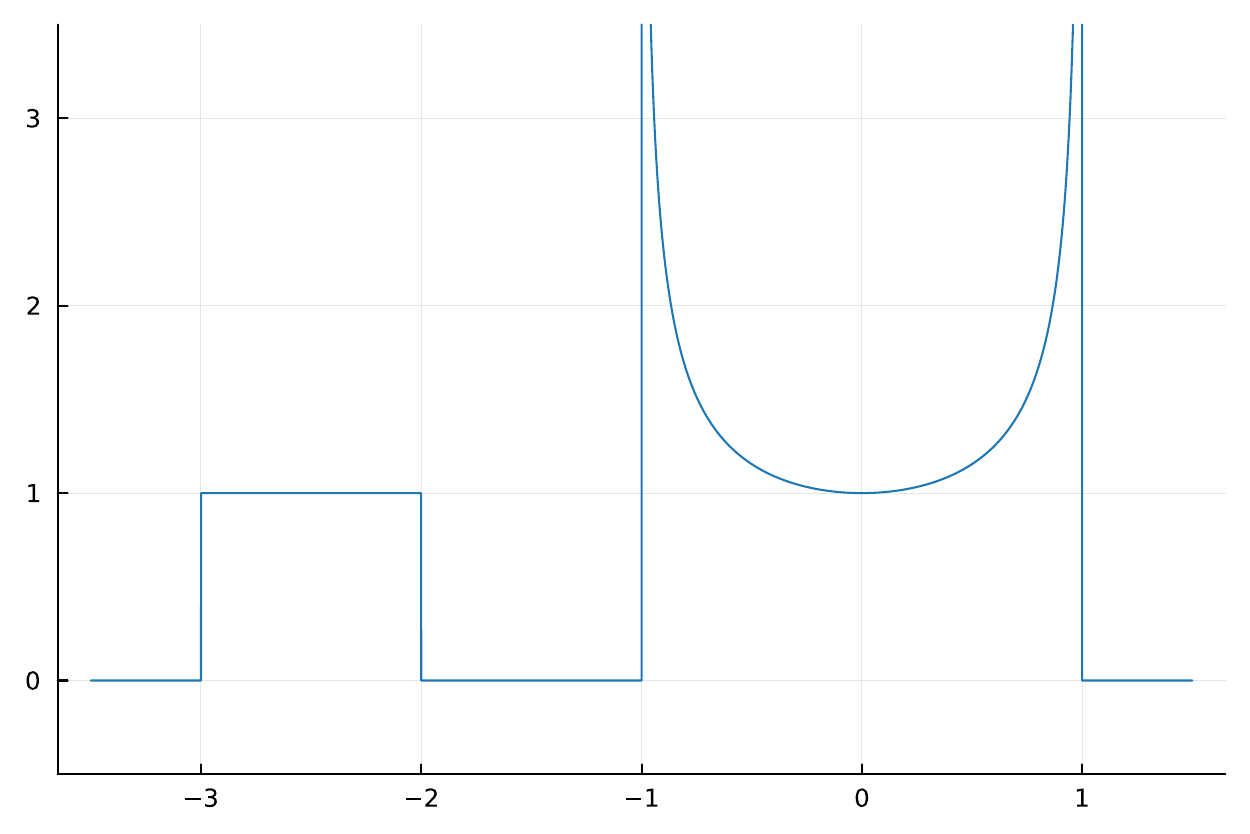}}
    \caption{The graph of $f$ (left) and its  {\graphname} (right).} 
    \label{fig:modifiedgraph}
\end{figure}

\begin{theorem}[Bounds on inverses of Stieltjes transform 2]
    \label{theorem:bounds2}
    Let $\mu$ be a measure with support contained in a compact interval $[a, b]$, such that $\mu$ has a density $\rho(x)$ whose whose restriction to $(a,b)$ lies in $W^{1,p}(a,b)$ for some $1<p<\infty$. Note that we allow for possible jump discontinuities at the endpoints $a$ and $b$.
    Then the number of solutions to $G_{\mu}(z) = \zeta$ is bounded above by
    \begin{equation*}
        N = \frac 12 \sup_{y > 0, y \notin S} \connectedcountersymbol\left(\set{x \in \R: (x,y) \in \graphsymbol(\rho)}\right)
    \end{equation*}
    where $\graphsymbol$ is the {\graphname} of $\rho(x)$ and $S$ is an arbitrary nowhere dense subset of $\R$.
\end{theorem}

\begin{proof}
    We approximate $\rho(x)$ with a sequence of densities $\rho_n(x)$ satisfying Theorem~\ref{theorem:bounds1} that converge pointwise almost everywhere to $\rho(x)$, much like in Example~\ref{example:uniformdist}.
    Define a density $\rho_n(x)$ as
    \begin{align}
        \label{eq:densitybound2}
        \rho_n(x) &= \frac 1{Z_n} \phi_n(x) \\
        \phi_n(x) &= \left\{\begin{matrix}
            n\rho(a+\frac 1n) (x-a) & a \leq x < a + \frac 1n \\
            \rho(x) & a + \frac 1n \leq x \leq b - \frac 1n\\
            n\rho(b-\frac 1n)(b-x)& b - \frac 1n < x \leq b\\
            0 & \text{otherwise}
           \end{matrix}\right.\notag
    \end{align}
    where $Z_n$ is a normalisation constant, \ie,
    
    \begin{equation*}
        Z_n = \int_{\R}\rho_n(x) \dd x
    \end{equation*}

    We now claim that $N$ then satisfies the following for all $n \in \Nat$
    \begin{equation}
        \label{eq:importantbound2}
        N \geq \frac 12 \sup_{y > 0, y \notin S} \connectedcountersymbol\left(\set{x \in \R: \phi_n(x) = y}\right)
    \end{equation}

    This inequality can be proven by considering the two intervals where $\phi_n$ and $\rho$ may differ. Recall that if a closed set $A \subset \R$ does not contain $x \in \R$, then $\connectedcountersymbol(A) = \connectedcountersymbol(A \cap (-\infty, x)) + \connectedcountersymbol(A \cap (x, \infty))$.
    
    Define the intervals $I_1 = [a, a+\frac 1n)$ and $I_2 = (b-\frac 1n, b]$ and let $y>0$ be arbitrary. There are two cases to consider first:
    \begin{itemize}
        \item If $\rho(a+ \frac 1n)> y$, then 
        no connected component of $A_n := \set{x \in \R: \phi_n(x) = y}$ or $A := \set{x \in \R: (x,y) \in \graphsymbol(\rho)}$ contains $a + \frac 1n$. 
        The set $A_n \cap I_1$ consists of a single element (it is the intersection of two straight lines), and hence exactly one connected component {\ie} $\connectedcountersymbol(A_n \cap I_1)=1$. However, by the path connectedness of the {\graphname} of $\rho(x)$ it follows that $A \cap I_1$ is non-empty and therefore $\connectedcountersymbol(A \cap I_1) \geq 1$.
        Hence we have
        \begin{align*}
            \connectedcountersymbol(A_n) = 1+ \connectedcountersymbol(A_n \setminus I_1), \hspace{1cm}
            \connectedcountersymbol(A) \geq 1 + \connectedcountersymbol(A \setminus I_1).
        \end{align*}

        \item If $\rho(a+ \frac 1n) \leq y$, then the set $A_n \cap I_1$ is empty. Thus, we have
        \begin{align*}
            \connectedcountersymbol(A_n) = \connectedcountersymbol(A_n \setminus I_1), \hspace{1cm}
            \connectedcountersymbol(A) \geq \connectedcountersymbol(A \setminus I_1).
        \end{align*}

    \end{itemize}
        We may perform the same argument with the discontinuity at $b$ with the sets $B_n = A_n \setminus I_1$, $B = A \setminus I_1$, and $I_2 = (b-\frac1n, b]$ to obtain that
        \begin{itemize}
            \item If $\rho(b-\frac1n) >y$, then 
            \begin{align*}
                \connectedcountersymbol(B_n) = 1+ \connectedcountersymbol(B_n \setminus I_2), \hspace{1cm}
                \connectedcountersymbol(B) \geq 1 + \connectedcountersymbol(B \setminus I_2).
            \end{align*}
            \item If $\rho(b-\frac1n)\leq y$, then
        \begin{align*}
            \connectedcountersymbol(B_n) = \connectedcountersymbol(B_n \setminus I_2), \hspace{1cm}
            \connectedcountersymbol(B) \geq \connectedcountersymbol(B \setminus I_2).
        \end{align*}
    \end{itemize}

        However, the sets $B_n \setminus I_2$ and $B \setminus I_2$ are exactly the same, since we can rewrite these sets as
        \begin{align*}
            B_n \setminus I_2 = A_n \setminus (I_1 \cup I_2)= \set{x \in (a+\frac 1n, b-\frac 1n): \phi_n(x) = y},\\
            B \setminus I_2 = A \setminus (I_1 \cup I_2) = \set{x \in  (a+\frac 1n, b-\frac 1n): (x,y) \in \graphsymbol(\rho)}.
        \end{align*}

        from which the equivalence follows since on the interval $(a+\frac 1n, b-\frac 1n)$, $\phi_n$ is defined to be same as $\rho$ \eqref{eq:densitybound2} and $\rho(x)$ is continuous. We therefore obtain
        \begin{equation*}
            \connectedcountersymbol(A) \geq \connectedcountersymbol(A_n)
        \end{equation*}
        and taking the supremum over all $y>0, y \notin S$ yields the proof of the inequality~\eqref{eq:importantbound2}.

    Since $N$ is an upper bound on the number of inverses for the sequence $\dd \mu_n = \rho_n(x) \dd x$, by Proposition~\ref{proposition:weakbound} it is also an upper bound for $\dd \mu = \rho(x)$ as well.
\end{proof}

\begin{remark}
    Theorem~\ref{theorem:bounds2} can be extended to measures supported on multiple intervals, or measures with discontinuities in their density in the interior of the support. The general idea of approximating jump discontinuities with affine slopes holds in any case.
\end{remark}

The above theorem now allows us to deal with absolutely continuous measures which may not decay with root-like behaviour near the boundary of the support, which includes measures like the uniform distribution or any measure whose density can be expanded in terms of Legendre polynomials, see Section \ref{section:computing}.
Lastly, we extend the theorem to cover measures with densities which may have unbounded behaviour, such as measures with inverse square-root singularities in their densities.
\begin{theorem}[Bounds on inverses of Stieltjes transform 3]
    \label{theorem:bounds3}
    Let $\mu$ be a measure supported on a finite union of disjoint compact intervals $[a_i, b_i]$, such that $\mu$ has a density $\rho(x)$ which is absolutely continuous on each interval of support. In addition, we suppose that the restriction of $\rho(x)$ to $I$ is in $W^{1,p}(I)$ for each open interval $I \subset \Gamma$ where $\Gamma$ is the support, with possibly unbounded behaviour at the boundary of $\Gamma$.
    Then the number of solutions to $G_{\mu}(z) = \zeta$ is bounded above by
    \begin{equation*}
        N = \frac 12 \sup_{y > 0, y \notin S} \connectedcountersymbol\left(\set{x \in \R: (x,y) \in \graphsymbol(\rho)}\right)
    \end{equation*}
    where $\graphsymbol$ is the {\graphname} of $\rho(x)$ and $S$ is an arbitrary nowhere dense subset of $\R$.
\end{theorem}

\begin{proof}
    Define truncated measures $\rho_n$ as
    \begin{align*}
        \rho_n(x) = \frac{1}{Z_n}\phi_n(x)
        \phi_n(x) = \min(n, \rho(x))
    \end{align*}
with $Z_n$ being a normalisation constant.    
    These densities converge pointwise, hence the corresponding measures $\mu_n$ converge weakly. By Theorem~\ref{theorem:bounds2}, since $\Nat$ is a nowhere dense subset of $\R$, we have an upper bound for the number of inverses and for every $n \in \Nat$, this bound satisfies
    \begin{align*}
        \frac 12 \sup_{y > 0, y \notin S \cup \Nat} \connectedcountersymbol\left(\set{x \in \R: (x,y) \in \graphsymbol(\phi_n)}\right) &=\frac 12 \sup_{y \in (0,n), y \notin S \cup \Nat} \connectedcountersymbol\left(\set{x \in \R: (x,y) \in \graphsymbol(\phi_n)}\right)\\
        &=\frac 12 \sup_{y \in (0,n), y \notin S \cup \Nat} \connectedcountersymbol\left(\set{x \in \R: (x,y) \in \graphsymbol(\rho)}\right) \\
        &\leq \frac 12 \sup_{y > 0, y \notin S} \connectedcountersymbol\left(\set{x \in \R: (x,y) \in \graphsymbol(\rho)}\right) = N
    \end{align*}
    where the second equality holds since the {\graphname} of $\rho$ and $\phi_n$ are the same whenever $y<n$. 
    Since $N$ is an upper bound on the number of inverses for the sequence $\dd \mu_n = \rho_n(x) \dd x$, by Proposition~\ref{proposition:weakbound} it is also an upper bound for $\dd \mu = \rho(x)$ as well.
\end{proof}

This allows us to know an upper bound on the number of solutions to measures whose density has negative-root behaviour near the boundary, such as the arcsine law $\dd \mu = \frac{\dd x}{\pi \sqrt{1-x^2}}$ or Jacobi-type measures when either exponent is negative. In the case of the arcsine law $\mu_{\text{arcsine}}$, the obtained upper bound on the number of solutions to $G_{\mu_{\text{arcsine}}}(z) = \zeta$, is 2. However we actually know that the arcsine law has univalent Stieltjes transform (the inverse is given by $G_{\text{arcsin}}^{-1}(\zeta) = \frac{\sqrt{\zeta^{2}+1}}{\zeta}$ with an appropriate branch cut). Hence, the bound is not sharp.

\begin{remark}
    \label{remark:jacobiweightunivalent}
    One imporant corollary of Theorem \ref{theorem:bounds3} is that the Stieltjes transform of a Jacobi weight $w(x) = (1-x)^{\alpha}(1+x)^{\beta} $when at least one of the exponents is non-negative is univalent. In the case that both exponents are negative, the above theorem only gives a bound of at most $2$ inverses. The Stieltjes transform is in fact univalent  for $\alpha=\beta=-0.5$ whilst it is not univalent for $\alpha=\beta=-0.7$, which we observe by numerically evaluating the inverses using the techniques in this paper.
\end{remark}

\section{Computing inverses of Stieltjes transforms and convergence}
\label{section:computing}

We will utilise an algorithm for finding zeros of analytic functions that originates from Kravanja, Barel, Sakurai~\cite{kravanjabarel2000}. Let $\Omega$ be a simply connected open set and a holomorphic function $f:\Omega \to \C$. Let $\gamma(t)$ be a closed contour. We wish to find \textit{all} roots of $f$ lying inside the contour, denoted $z_i$.

Define the \textit{moments} 
\begin{equation*}
    s_n = \frac{1}{2\pi \I } \oint_{\gamma} z^n \frac{1}{f(z)}\dd z 
\end{equation*}
and consider the two Hankel matrices
\begin{align}
\label{eq:hankelmatrices}
    H_{0,n}=\begin{pmatrix}
        s_0 & s_1 & \cdots  &  s_{n-1}\\
        s_1 &  \ddots &  &  \vdots\\
        \vdots &  &  \ddots & \vdots \\
            s_{n-1} &  \cdots&  \cdots&  s_{2n-2}\\
        \end{pmatrix},
    \hspace{1cm}
    H_{1,n}=\begin{pmatrix}
        s_1 & s_2 & \cdots  &  s_{n}\\
        s_2 &  \ddots &  &  \vdots\\
        \vdots &  &  \ddots & \vdots \\
            s_{n} &  \cdots&  \cdots&  s_{2n-1}\\
        \end{pmatrix}.
\end{align}
Then we have the following
\begin{proposition}[Theorem 1.6.5 in~\cite{kravanjabarel2000}]
    \label{proposition:eigenvalueroots}
    Let there be $N$ roots of the function $f$ inside the contour $\gamma$ and let $M \geq N$. Then the eigenvalues of the generalised eigenvalue problem $H_{1,M} - \lambda H_{0,M}$ consist of the roots $z_1 \dots z_N$ followed by $M-N$ arbitrary values.
\end{proposition}

Given $f$ analytic on the closed unit disc $\overline{\D}$ with roots $z_1, \dots, z_N$ inside $\overline{\D}$, we will define the polynomial $P_N$ as
\begin{equation*}
    P_N(z) = \prod_{j=1}^{N} (z-z_j)
\end{equation*}
Then, the function $g(z) = \frac{P_N(z)}{f(z)}$ is analytic on $\overline{\D}$ and we can define $g_{N-1}$ as the unique polynomial of degree $N-1$ that interpolates $g$ at the roots $z_i$. Then it can be shown that 
\begin{equation*}
    s_n = \frac{1}{2\pi \I } \oint_{\gamma} z^n \frac{g_{N-1}}{P_N(z)}\dd z 
\end{equation*}

The integrals for $s_n$ are approximated using the trapezium rule, which has exponential rate of convergence to the true value of the integral. This results in a similar rate of convergence for the generalised eigenvalue problem, as seen in Proposition~\ref{proposition:ksv2}.
More specifically, let $\omega_j = e^{\frac{2 \pi \I j}{K}}$ and define
\begin{equation*}
    \hat{s}_n = \frac{1}{K} \sum_{j=0}^{K-1} \omega_j^n \frac{\omega_j}{f(\omega_j)} \hspace{1cm} 
    \tilde{s}_n = \frac{1}{K} \sum_{j=0}^{K-1} \omega_j^n \frac{\omega_j g_{N-1}(\omega_j)}{P_N(\omega_j)}
\end{equation*}
Lastly, we define matrices $\hat{H}_{i,n}$ and $\tilde{H}_{i,n}$ for $i=0,1$ analogously to \eqref{eq:hankelmatrices}.

\begin{proposition}[Section 3 in~\cite{Kravanja2003}]
    \label{proposition:ksv2}
    Suppose that $\gamma$ is the unit circle. Let $N$ the number of roots and $K$ be the number of quadrature points such that $K \geq 2N$. Then we have
    \begin{equation*}
        \hat{H}_{1,N} - \lambda \hat{H}_{0,N} = \tilde{H}_{1,N} - \lambda \tilde{H}_{0,N} + O(r^{2N - K})
    \end{equation*}
    where $r>1$ is the largest radius such that $\frac{1}{f}$ is holomorphic outside the closed unit disc. Moreover, the eigenvalues of the pencil $\tilde{H}_{1,N} - \lambda \tilde{H}_{0,N}$ are exactly the roots of $f$.
\end{proposition}
Note that the above proposition only applies when $N$ is the number of roots, however it can be checked that the same result holds if instead we have $M \geq N$ and considered the pencils $\hat{H}_{1,M} - \lambda \hat{H}_{0,M}$ and $\tilde{H}_{1,M} - \lambda \tilde{H}_{0,M}$. The only difference is that the eigenvalues of $\tilde{H}_{1,M} - \lambda \tilde{H}_{0,M}$ are described as in Proposition \ref{proposition:eigenvalueroots}.

\subsection{Conformal mappings}
Consider a Borel probability measure $\mu$ with compact support $\Gamma$ that is a union of finitely many closed intervals $[a_i, b_i]$ and has convex hull $[a,b]$. We will first consider the location of all zeros lying in $\C \setminus [a,b]$. This can be done by considering the function $G_{\mu}(M_{(a,b)}(J(z)))$ which maps from the unit disc to $\C \setminus [a,b]$. This expression is not well defined at $z=0$, however the singularity is removable and thus is analytic in the open unit disc.
\begin{lemma}
    $G_{\mu}(M_{(a,b)}(J(z)))$ has a removable singularity at $z=0$ and the limiting value near $z=0$ is $0$. In particular, the function is analytic on the interior of the unit disc $\D$.
\end{lemma}
\begin{proof}
    Recall the asymptotic behaviour for $G_{\mu}(z) = \frac 1z + O(\frac{1}{z^2})$. Then it is clear that the limit as $z \to 0$ exists and is $0$.
\end{proof}
Thus, we can find all roots inside a circular contour of radius $r \in (0,1)$ of $G_{\mu}(M_{(a,b)}(J(z)))$, which corresponds to finding all roots lying \textit{outside} an ellipse surrounding the support $\Gamma$ (Figure~\ref{fig:contours}). As $r$ tends to $1$, the ellipse shrinks closer to the interval $[a,b]$, which results in more inverses being recovered. However, the presence of singularities near the support generally slows the rate of convergence of numerical methods, see Examples in Section~\ref{subsection:numericalexperiments1}.

If the support $\Gamma$ is not connected, then the above transformation will not be able to recover inverses lying in $[a,b] \setminus \Gamma$. However, one can choose a transformation such as $G_{\mu}(M_{(c,d)}(z))$ to recover the inverses lying between an interval $c$ and $d$. Combined with the previous map involving the Joukowski transform, all inverses can be found.

\begin{figure}[H]
    \centering
    \subfloat[]{\includegraphics[width = 0.5\textwidth]{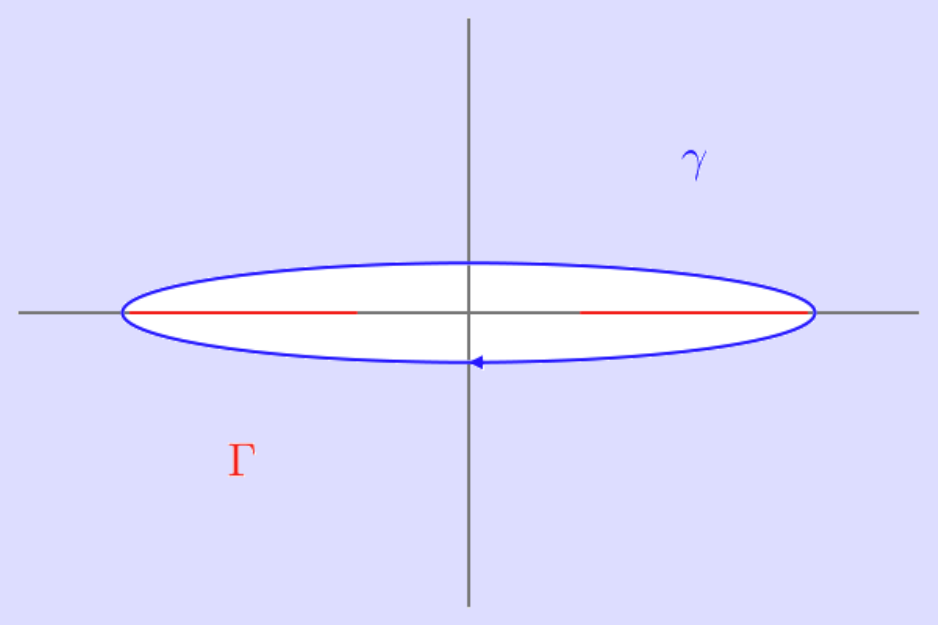}}
    \subfloat[]{\includegraphics[width = 0.5\textwidth]{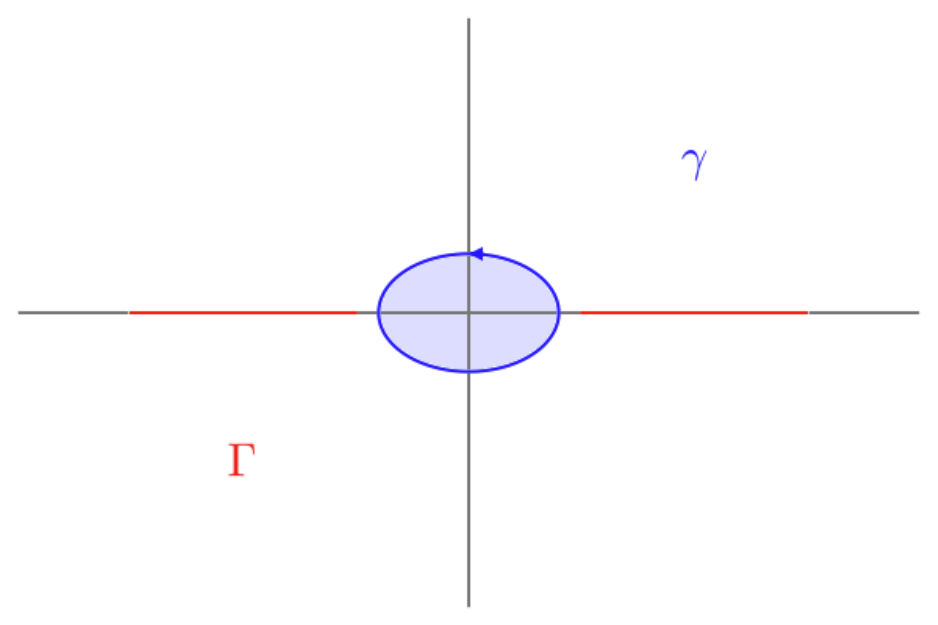}}
    \caption{Region where inverses can be recovered using $M_{(a,b)}(J(z))$ (Left) or $M_{(c,d)}(z)$ (Right), shaded in blue. As $r$ increases up to $1$, the union of the blue regions is the whole domain of $G_{\mu}$. }
    \label{fig:contours}
\end{figure}
\subsection{Convergence of orthogonal polynomial expansions in inverse Stieltjes transforms}
In practice, for absolutely continuous measures, the Stieltjes transform is computed in terms of orthogonal polynomial expansions. Given a positive weight function $w(x)$ with support on $\Gamma$, let $p_n$ be the associated orthogonal polynomials and let $a_n,b_n,c_n$ be the three-term recurrence coefficients.
\begin{align}
    &xp_0(x) = a_0 p_0(x) + b_0 p_1(x) \\
    &xp_n(x) = c_{n-1} p_{n-1}(x) + a_n p_n(x) + b_n p_{n+1}(x)
    \hspace{0.5cm} n \geq 1
\end{align}
We define the Stieltjes transforms of weighted orthogonal polynomials as
\begin{equation*}
    q_n(z) = \int_{\Gamma}\frac{p_n(x)w(x)}{z-x}\dd x.
\end{equation*}
It is well known that $q_n$ satisfy a very similar three-term recurrence relation as the orthogonal polynomials~\cite{gautschi1981b}:
\begin{align}
    &zq_0(z) = a_0 q_0(z) + b_0 q_1(z) + \int_S p_0(x)w(x)\dd x\\
    &zq_n(z) = c_{n-1} q_{n-1}(z) + a_n q_n(z) + b_n q_{n+1}(z),
    \hspace{0.7cm} n \geq 1,
\end{align}
which enables fast computation of Stieltjes transforms.
Given a measure with density $\dd \mu = \rho(x) \dd x$, one can decompose $\rho(x)$ into
\begin{equation*}
    \rho(x) = r(x)w(x)
\end{equation*}
where $w(x)$ is an appropriate weight function with support on $\Gamma$, and $r(x)$ is a bounded function. Under these assumptions, we can expand $r(x)$ in the orthogonal polynomial basis $p_n$ with coefficients $\phi = (\phi_0, \phi_1, \dots)$ defined in terms of the inner product induced by $w(x)$, so that we have:

\begin{equation}
    r(x) = \sum_{n=0}^{\infty} \phi_n p_n(x) \hspace{1cm} \phi_n = \frac{\left \langle r(x), p_n(x) \right \rangle}{\left \| p_n(x) \right \|^2}
\end{equation}

The coefficients $\phi_n$ decay at a rate that depends on the smoothness/analycity of the function $r(x)$, with faster decay when $r(x)$ is more times differentiable (Some standard references are~\cite{szegő1975orthogonal}).
We will consider the truncation of this series, by choosing an integer $m$ such that:
\begin{equation}
\label{eq:f r_m}
    r(x) \approx r_m(x) := \sum_{n=0}^{m} \phi_n p_n(x)
    \hspace{1cm} \textup{and} \hspace{1cm}
    \phi_{m} \neq 0
\end{equation}
We can therefore define approximations of the Stieltjes transform of $\mu$ as functions
\begin{equation*}
    \label{eq:truncatedcauchytransform}
    G_{\mu,m}(z) = \int_{\Gamma}\frac{r_m(x)}{z-x}w(x) \dd x = \sum_{n=0}^{m} \phi_n q_n(z).
\end{equation*}
Note that $G_{\mu,m}$ is not necessarily the Stieltjes transform of some positive measure, as $r_m(x)$ may  be negative. We will study the convergence of the roots of $G_{\mu,m}(z) - \zeta$ to the true solutions of $G_{\mu}(z) = \zeta$. Recalling that the $w(x)$ weighted $L^2$ norm of $r_m(x) - r(x)$ goes to $0$ as $m$ tends to infinity, we have the following.

\begin{lemma}
    \label{lemma:hurwitzconvergence}
    The number of solutions to $G_{\mu, m}(z) = \zeta$ and $G_{\mu}(z) = \zeta$ are the same for large enough $m$. Additionally, the roots of $G_{\mu, m}(z) = \zeta$ converge to those of $G_{\mu}(z) = \zeta$ as $m \rightarrow \infty$.
\end{lemma}
\begin{proof}
    We first show that the sequence of analytic functions $G_{\mu, m}$ converge uniformly on all compact subsets of $\C \setminus \Gamma$ to $G_{\mu}$.
    Let $K$ be any compact subset of $\C \setminus \Gamma$. We then have, for every $z \in k$
    \begin{equation}
    \label{eq:|gm-g|}
        \left|G_{\mu, m}(z) - G_{\mu}(z)\right| = \left|\int_{\Gamma}\frac{r_m(x)-r(x)}{z-x}w(x)\dd x\right| \leq C\int_{\Gamma} |r_m(x) - r(x)| w(x) \dd x
    \end{equation}
    where $C>0$ depends on $\inf \set{|z-x|: x \in \Gamma, z \in K}$.
    By the Cauchy-Schwarz inequality, we have:
    \begin{equation}
        \int_{\Gamma} |r_m(x) - r(x)|w(x) dx \leq \left( \int_{\Gamma} |r_m(x) - r(x)|^2w(x) dx \right)^{\frac 12} \left( \int_{\Gamma} w(x) dx \right)^{\frac 12} 
    \end{equation}
    Since $w(x)$ is integrable, the right hand side goes to $0$ uniformly in $K$. Therefore, $G_{\mu, m}$ converges uniformly on compact subsets to $G_{\mu}$. An application of Hurwitz's theorem finishes the proof.
\end{proof}
We can obtain a quantitative rate of convergence by applying Rouche's theorem in a small radius around each root. We will need the following lemma.

\begin{lemma}
    \label{lemma:rootdistance}
    Suppose that $f$ and $g$ are complex analytic functions on an open set $U$. Let $D = B_{d,z_0}$ be an open disc centered at $z_0$ with radius $d$ whose closure is contained within $U$. Suppose that $f$ has a zero of order $k$ at $z_0$ and no other zeros inside $D$. Then, for any $\varepsilon \in (0, d)$, the number of zeros of $g$ inside $B_{\varepsilon,z_0}$ is $k$ if
    \begin{align*}
       \sup_{|w-z_0| = \varepsilon}|f(w)-g(w)| < \frac{|f^{(k)}(z_0)|}{k!}\varepsilon^k - M_{d,z_0}\frac{b^{k+1}}{1-b} 
    \end{align*}
where $b = \frac{\varepsilon}{d} < 1$ and $M_{d,z_0} = \max_{|w-z_0| = d} |f(w)|.$ 
    In particular, the distance of the roots of $g$ to $z_0$ is $O(\delta^{\frac{1}{k}}$), where $\delta = \sup_{|w-z_0| = \varepsilon}|f(w)-g(w)|$.
\end{lemma}
\begin{proof}
    Recall Taylor's theorem for complex analytic functions in a closed disc $\overline{B_{d,z_0}}$.
    \begin{align*}
        f(z) = f(z_0) + f'(z_0)(z-z_0) + \cdots + \frac{f^{(k)}(z_0)}{k!}(z-z_0)^k + R_{k}(z) 
    \end{align*}
    where $|R_k(z)| \leq M_{d, z_0}\frac{b^{k+1}}{1-b}$.
    If $f$ has a zero of order $k$ at $z_0$, this reduces to
    \begin{align*}
        f(z) = \frac{f^{(k)}(z_0)}{k!}(z-z_0)^k + R_{k}(z)
    \end{align*}
    Let $\gamma(t) = \varepsilon e^{\I t} + z_0$ be a closed circular contour. Then for $z$ lying on the curve $\gamma$, we have the following lower bound independent of $z$:
    \begin{equation*}
        |f(z)| \geq \frac{|f^{(k)}(z_0)|}{k!}\varepsilon^k-M_{d, z_0}\frac{b^{k+1}}{1-b}.
    \end{equation*}
Thus the result follows from Rouche's theorem.
\end{proof}
\begin{theorem}
    \label{theorem:quantitativeconvergence}
    Let $\mu$ and $\mu_m$ be measures all with the same support $\Gamma$ and densities $\rho(x)$ and $\rho_m(x)$. Suppose that $\rho_m$ converges in $L^1(\Gamma)$ to $\rho$ as $m\to\infty$. Let $z_0$ be a solution of $G_{\mu}(z) = \zeta$ and let $d>0$ be such that $z_0$ is the only root in the open disc $B_{d, z_0} \subset \C \setminus \Gamma$ and has order $k$. Then for any $\varepsilon \in (0,d)$, the number of zeros of $G_{\mu,m}$ inside of $B_{\varepsilon,z_0}$ is $k$ if
    \begin{align*}
        \frac{1}{\mathrm{dist}(z_0,\Gamma)-\varepsilon} \|\rho_m - \rho\|_{L^1(\Gamma)} < \frac{|G_{\mu}^{(k)}(z_0)|}{k!}\varepsilon^k - M_{d,z_0}\frac{b^{k+1}}{1-b} 
     \end{align*}
where $b = \frac{\varepsilon}{d} < 1$ and $M_{d,z_0} = \max_{|w-z_0| = d} |G_{\mu}(w) - \zeta|$.
     In particular, for fixed $z_0$, the distance of the roots of $G_{\mu,m}$ to $z_0$ is $O(\|\rho_m - \rho\|_{L^1(\Gamma)}^{\frac 1k})$.
\end{theorem}
\begin{proof}
    For $w$ such that $|w - z_0| = \varepsilon$, we have the following bound
    \begin{equation*}
        \left|G_{\mu, m}(w) - G_{\mu}(w)\right| \leq \int_{\Gamma}\frac{|\rho_m(x) - \rho(x)|}{|w-x|}\dd x \leq \frac{1}{\mathrm{dist}(z_0,\Gamma)-\varepsilon} \|\rho_m - \rho\|_{L^1(\Gamma)}
    \end{equation*}
    The result then follows from Lemma~\ref{lemma:rootdistance}.
\end{proof}
\begin{remark}
    In general, the case where $G_{\mu,m}(z) = \zeta$ has roots which are not simple is rare in the sense that it only occurs at an isolated set of points where $G_{\mu}'(z) = 0$. Thus, in general the rate of convergence of the roots is $O(\|\rho_m - \rho\|_{L^1(\Gamma)})$.
\end{remark}
With Theorem~\ref{theorem:quantitativeconvergence}, we obtain that the rate of convergence depends on the $L^1$ distance between the densities. In the case of orthogonal polynomial expansions, the $L^1$ norm can be controlled in terms of the $L^2$ norm of the coefficients similar to the last step of Lemma~\ref{lemma:hurwitzconvergence}. More specifically, if our orthogonal polynomial sequence is uniformly bounded in $L^2$ norm, {\ie}, $M = \sup_{n} \|p_n(x)\|< \infty$,
then we have
\begin{align*}
    \|\rho_m - \rho\|_{L^1(\Gamma)} &\leq \left(\int_{\Gamma}|r_m(x) - r(x)|^2w(x)\dd x\right)^{\frac 12} \left(\int_{\Gamma}w(x)\dd x\right)^{\frac 12} \\
    &\leq C \| \phi - \phi_{(m)}\|_{\ell^2}
\end{align*}
where $\phi_{(m)} = (\phi_0, \dots, \phi_m, 0 \dots)$ and $C$ is a constant dependent only on $M$ and $w$.

\section{Numerical Experiments}
\label{section:numericalexperiments}
The following examples  numerically demonstrate the validity of the theorems presented in Sections 3 and 4, alongside measuring the rates of convergence.
All computations were done using Julia 1.10.4 in double precision arithmetic.

\subsection{Jacobi measures}
\label{subsection:numericalexperiments1}
Consider the Jacobi measure defined with support on $[-1,1]$ and density
\begin{equation*}
    \rho(x) = \frac{5}{16}(1+14x^2) (1-x^2)^2.
\end{equation*}
Such a measure is called a Jacobi measure due to its association with the Jacobi weight $w(x) = (1-x)^2(1+x)^2$. Roots are computed by finding the roots of $G_{\mu}(M_{(-1,1)}(J(z)))$ in a disc centered at $0$ with radius $r=0.99$. This corresponds to finding roots lying outside of an ellipse surrounding $[-1,1]$.

We first test $4$ different values of $\zeta$ and compute the error as $|G_{\mu}(\hat{z}_{i,j}) - \zeta_i|$ for each computed root $\hat{z}_{i,j}$ of $\zeta_i$. By plotting the density (Figure~\ref{fig:densitygammacurve1}), we obtain an upper bound on the number of solutions as $N=2$ via Theorem~\ref{theorem:bounds1}. We may also numerically verify the ideas used in the proof of Proposition~\ref{proposition:boundoffcurve} by plotting the curve $\gamma_1$ and computing the inverses.

\begin{figure}[H]
    \centering
    \subfloat[]{\includegraphics[width = 0.5\textwidth]{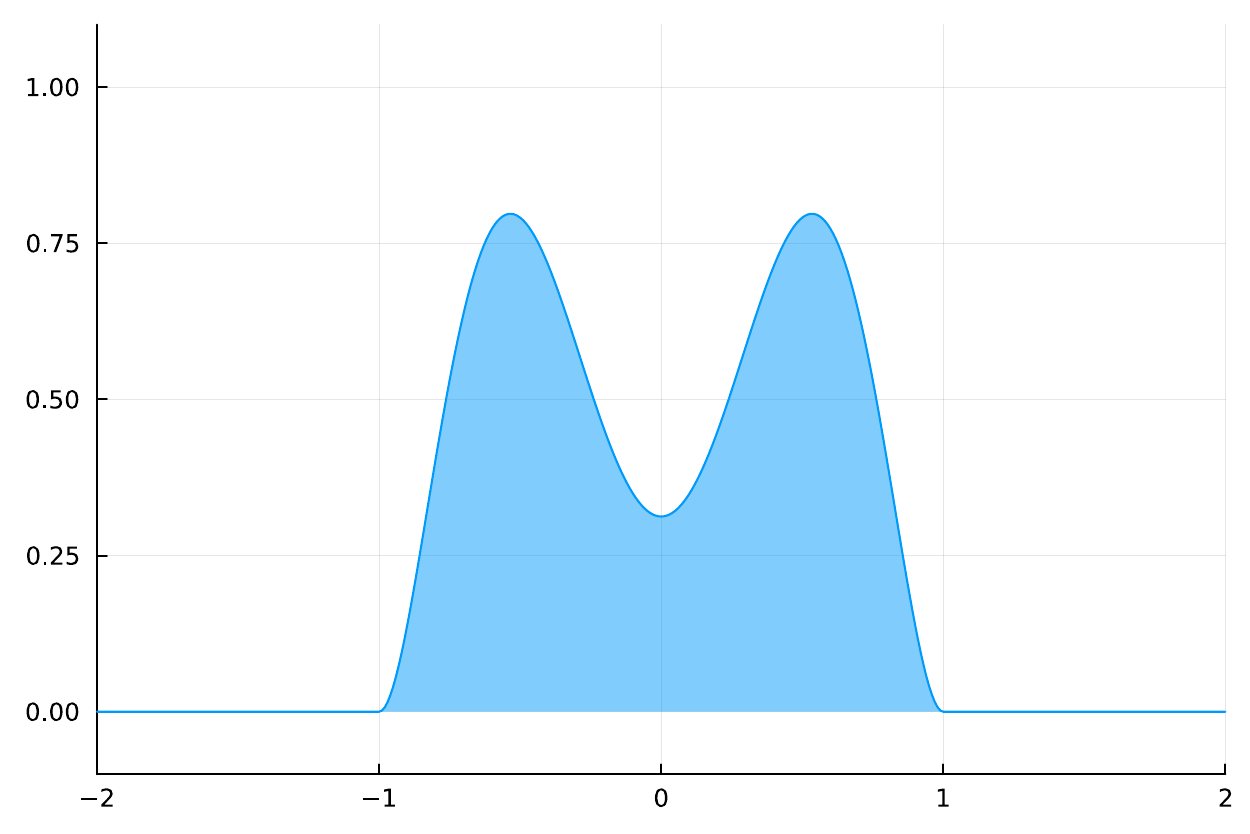}}
    \subfloat[]{\includegraphics[width = 0.5\textwidth]{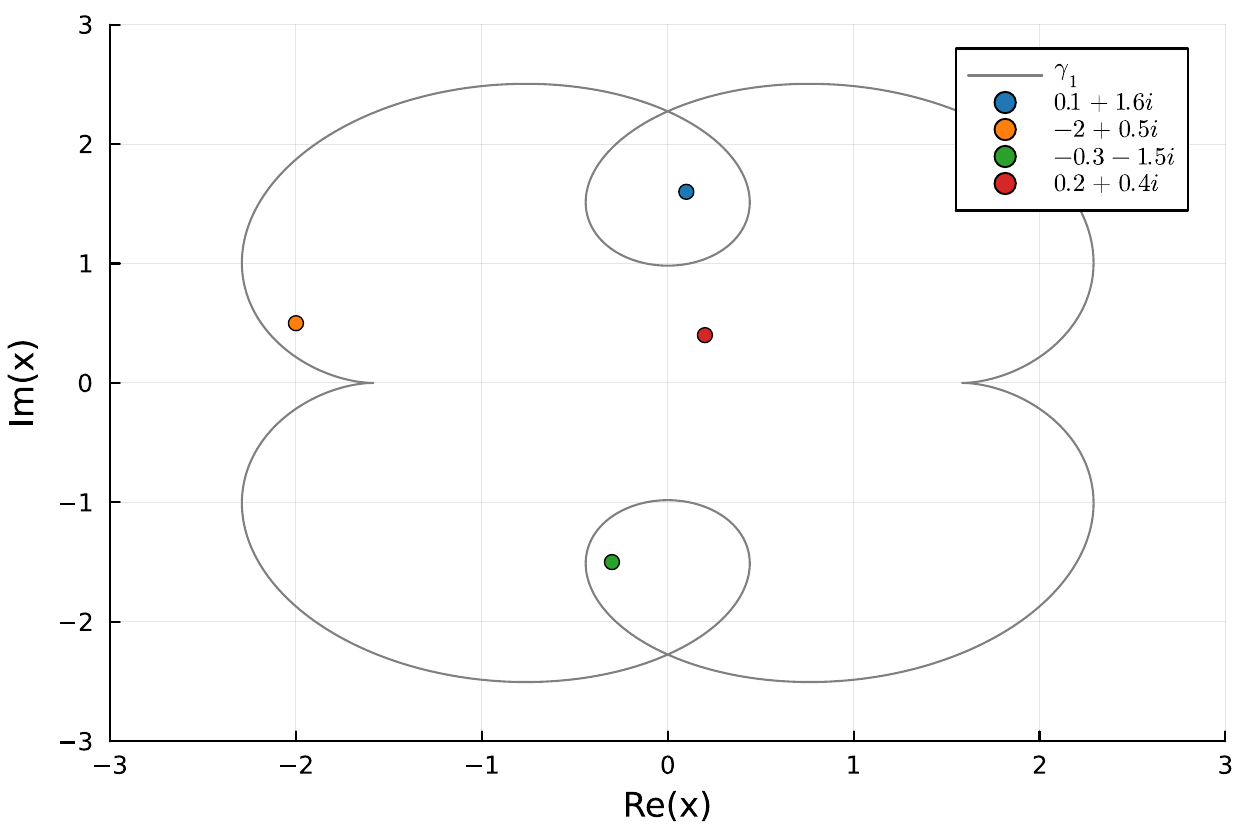}}
    \caption{Left: density of $\mu$ with Jacobi weight. Right: $4$ different $\zeta$ values and the curve $\gamma_1$.} 
    \label{fig:densitygammacurve1}
\end{figure}

We compute the inverses using values of $N=20,40,\dots1000$ with a fixed $r=0.99$, showing that for $\zeta = 0.1+1.6\I $ and $\zeta = -0.3-1.5\I $, we obtain two solutions to $G_{\mu}(z) = \zeta$ whilst for the other two values of $\zeta$, we obtain one solution (Figure~\ref{fig:error1}).

By changing the value of $r$ to be smaller, one can obtain faster convergence, as seen in Proposition~\ref{proposition:ksv2}. Generally, convergence slows down for $r$ close to $1$ due to the presence of singularities which get mapped to the boundary of the unit disc. However, for lower values of $r$ we actually observe that the error is actually \textit{increasing} in some roots.

The reason is that these roots do not lie in the region where inverses can be recovered, {\ie}, they lie \textit{inside} the ellipse (Figure~\ref{fig:ellipse1}). However, for smaller values of $K$, the number of quadrature points is insufficient to determine whether a root lies inside the contour. We see that for $r=0.9$, three of the inverses lie inside the ellipse. For larger values of $r$, all inverses lie inside the ellipse for $r=0.99$ and the roots converge as $K$ increases.

\begin{figure}[H]
    \centering
    \subfloat[]{\includegraphics[width = 0.5\textwidth]{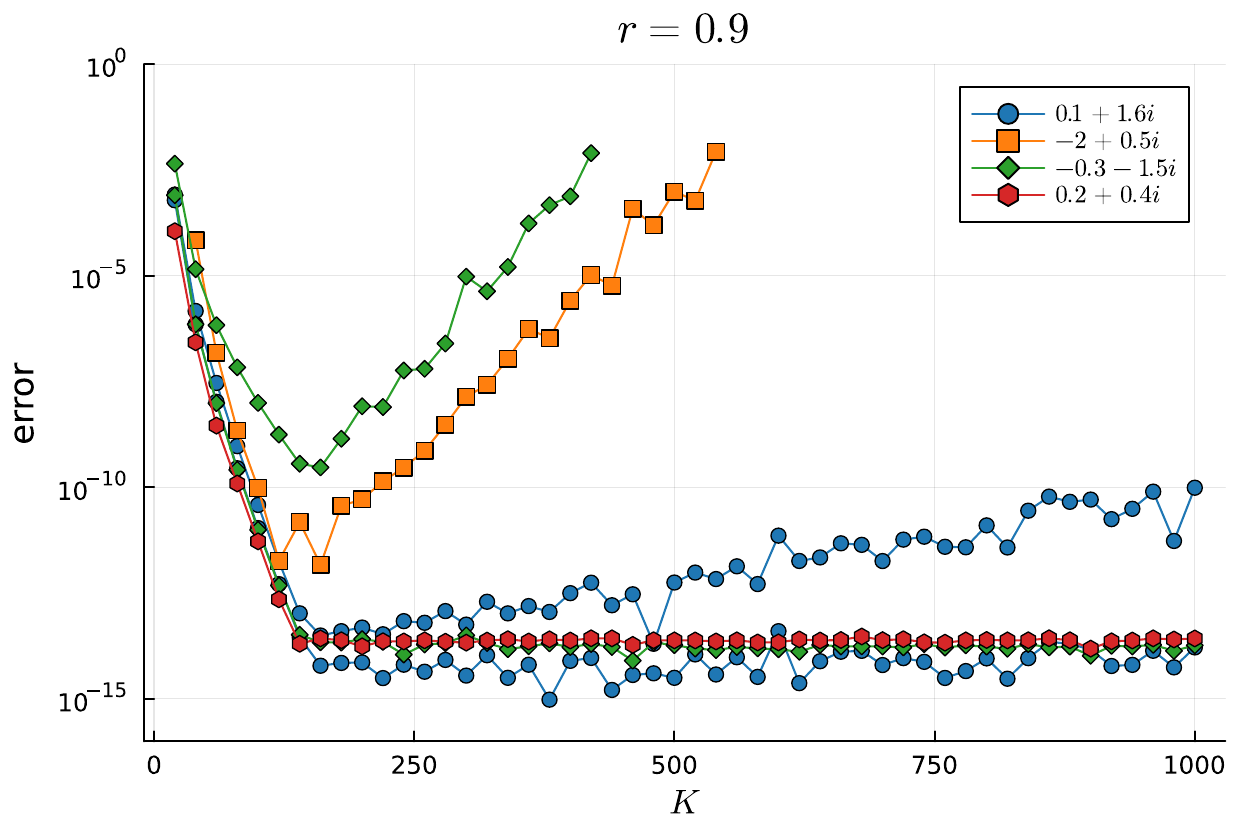}}
    \subfloat[]{\includegraphics[width = 0.5\textwidth]{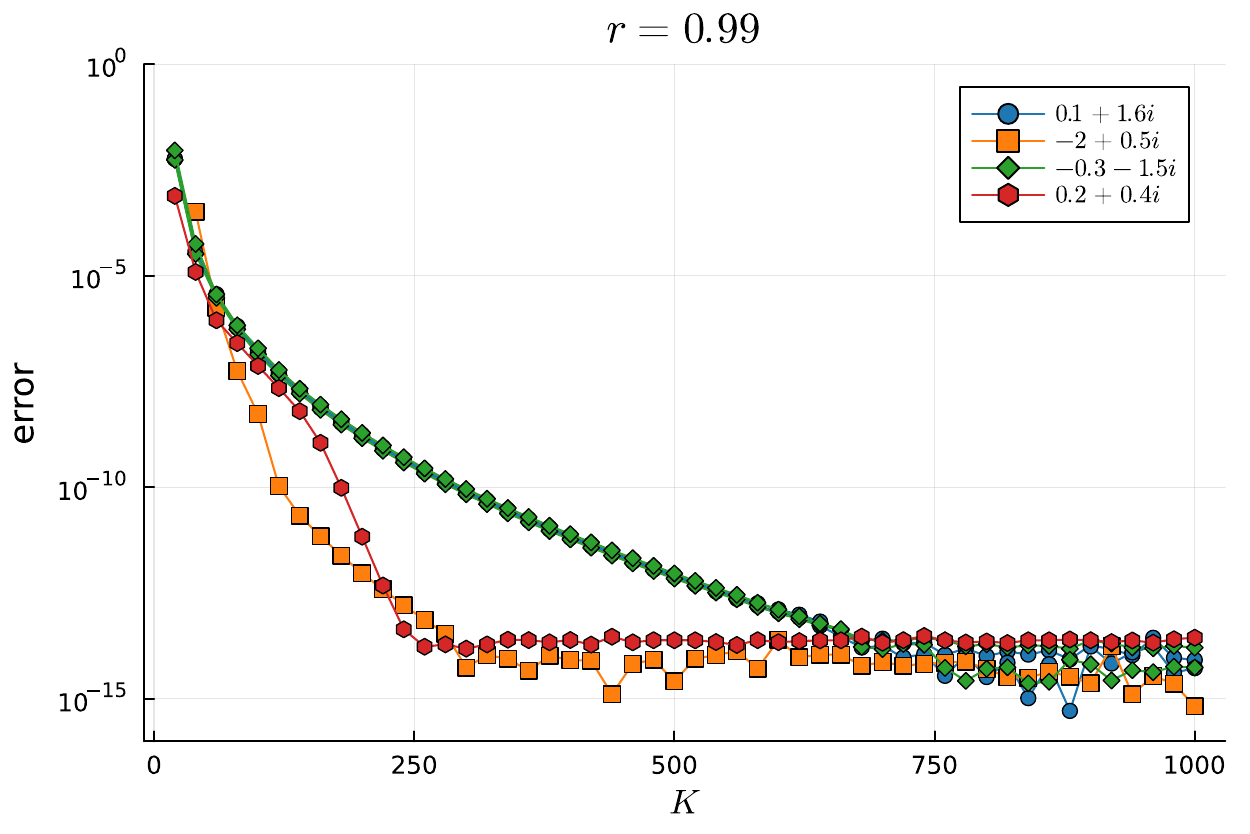}}

    \caption{Error computed as $|G_{\mu}(\hat{z}_{i,j}) - \zeta_i|$ plotted against number of quadrature points $K$, for $r=0.9$ and $r=0.99$. The points $0.1+1.6\I $ and $-0.3-1.5\I $ both have $2$ inverses, whilst $-2+0.5\I $ and $0.2+0.4\I $ each only have $1$. The errors that grow in the left figure correspond to roots that lie \textit{outside} the contour.}
    \label{fig:error1}
\end{figure}
\begin{figure}[H]
    \centering
    \subfloat[]{\includegraphics[width = 0.5\textwidth]{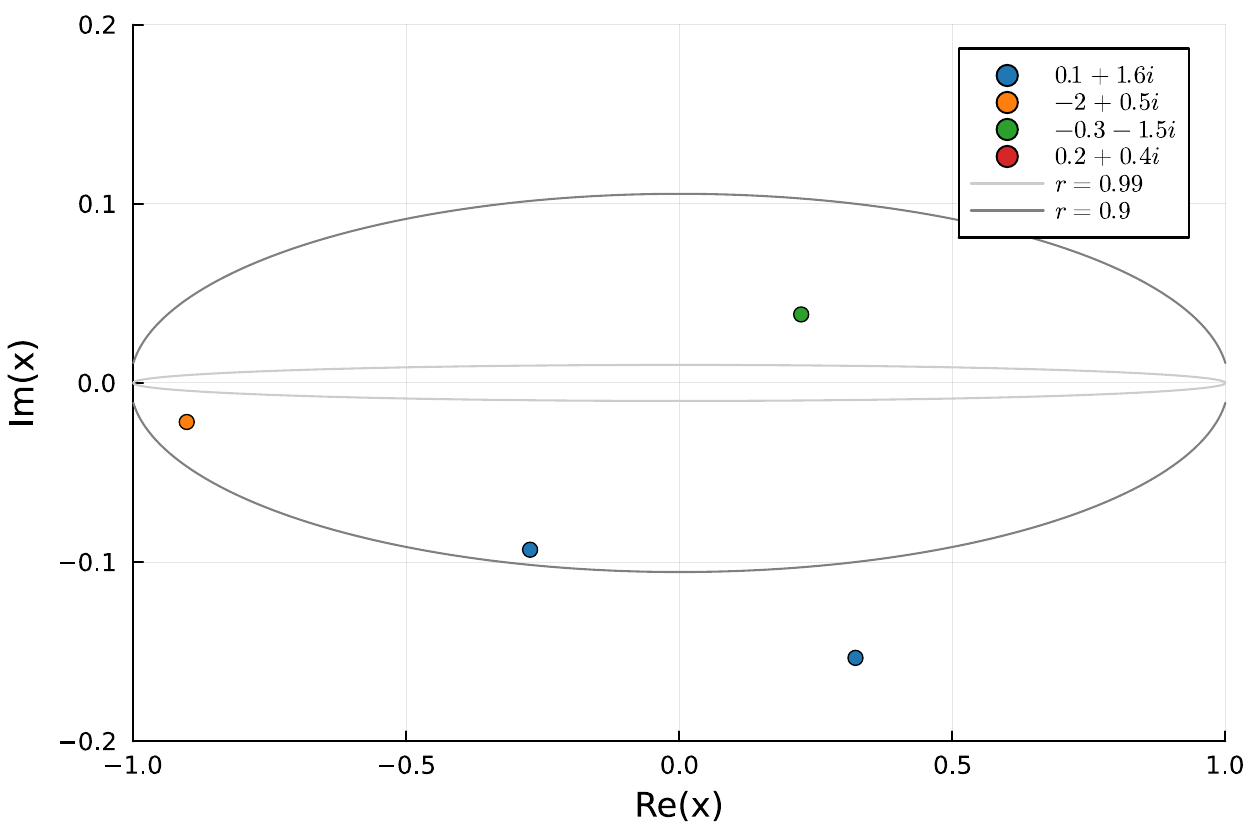}}
    \caption{Solutions of $G_{\mu}(z) = \zeta$ alongside the ellipses for $r=0.9$ and $r=0.99$.}
    \label{fig:ellipse1}
\end{figure}

\subsection{Measures supported on multiple intervals}
We now demonstrate the case of multiply-supported measures. Consider the multi-cut Jacobi measure defined by
\label{subsection:numericalexperiments2}
\begin{align*}
    \dd \mu_1 &= \frac{1}{Z_1} \left(x+3\right)^{-\frac{1}{3}}\left(-1-x\right)^{-\frac{2}{3}} \dd x &\text{ on } &x \in [-3,-1] \\
    \dd \mu_2 &= \frac{1}{Z_2} \left(x-1\right)^{\frac{1}{2}}\left(3-x\right)^{\frac{1}{2}}\left(2+\sin x\right) \dd x &\text{ on } &x \in [1,3]\\
    \mu &= \frac 12 \mu_1 + \frac 12 \mu_2
\end{align*}
where $Z_1$ and $Z_2$ are constants which normalise $\mu_1$ and $\mu_2$ to be probability measures. The resulting measure $\mu$ is supported on two disjoint intervals, whose density has unbounded behaviour near $x=-3$ and $x=-1$ (Figure~\ref{fig:densitygammacurve2}).

Theorem~\ref{theorem:bounds3} applies and yields a bound of $N=3$. (It turns out that this bound is achieved, though only in a very small region of the complex plane including $-0.5+0.65\I $.) Though the curve $\gamma_1$ is no longer well defined, we may still plot the curve $\gamma_{1-\varepsilon}$ where $\varepsilon>0$ is small and observe winding numbers.

\begin{figure}[H]
    \centering
    \subfloat[]{\includegraphics[width = 0.5\textwidth]{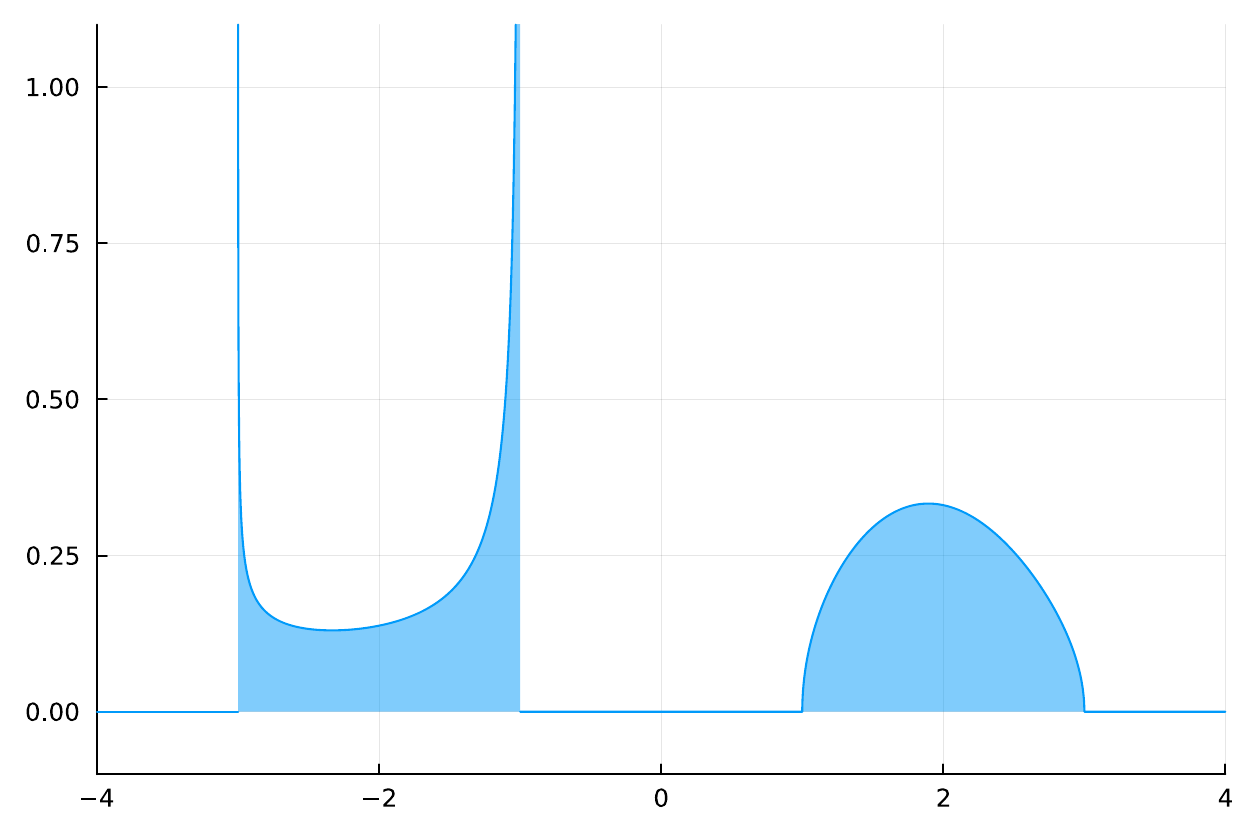}}
    \subfloat[]{\includegraphics[width = 0.5\textwidth]{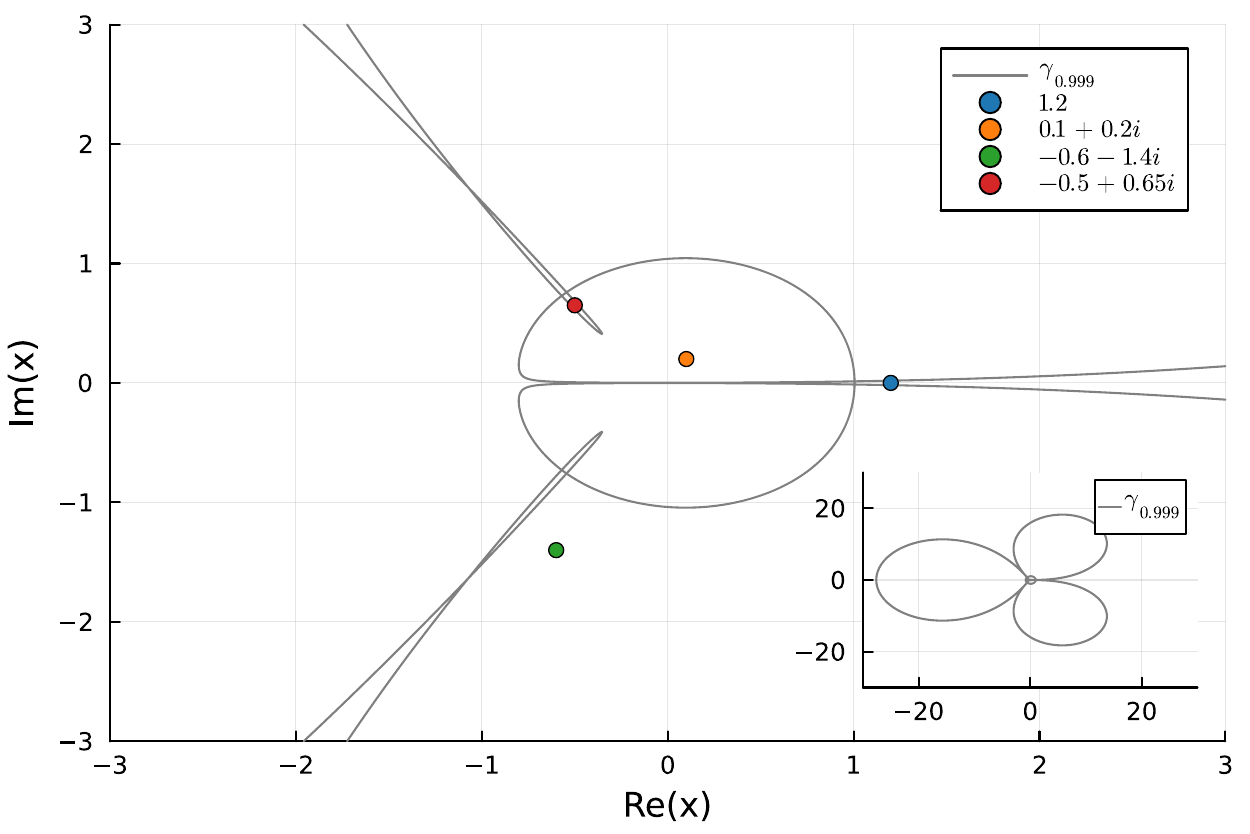}}
    \caption{Left: density of $\mu$ on multiple intervals. Right: $3$ different $\zeta$ values and the curve $\gamma_{0.999}$, with the same curve inset zoomed out.} 
    \label{fig:densitygammacurve2}
\end{figure}

By observing the winding numbers in Figure~\ref{fig:densitygammacurve2}, we can expect $2$ inverses for $0.1+0.2\I $ and $1$ inverse for $0.6-1.4\I $. In particular, the point $1.2$ does not lie inside of the curve $\gamma_r$ for any $r \in (0,1)$ and as such, does not appear to have any inverses. This is because the only inverse of $1.2$ lies in the interval $[-1,1]$, which is not covered by the Joukowski transform. However, we can still recover inverses by integrating in a circle of radius $r$ centered at $0$, as seen in Figure~\ref{fig:contours}.

We compute the inverses using values of $N=80,160,\dots,4000$ with a fixed $r=0.99$. In general, for measures with unbounded density we require more quadrature points due to the unboundedness of $\gamma_r$ as $r$ approaches $1$. We also recover all $3$ inverses of the point $\zeta=-0.5+0.65\I $ using $r=0.999$ and a high number of quadrature points to demonstrate numerically the presence of $3$ inverses.

\begin{figure}[H]
    \centering
    \subfloat[]{\includegraphics[width = 0.5\textwidth]{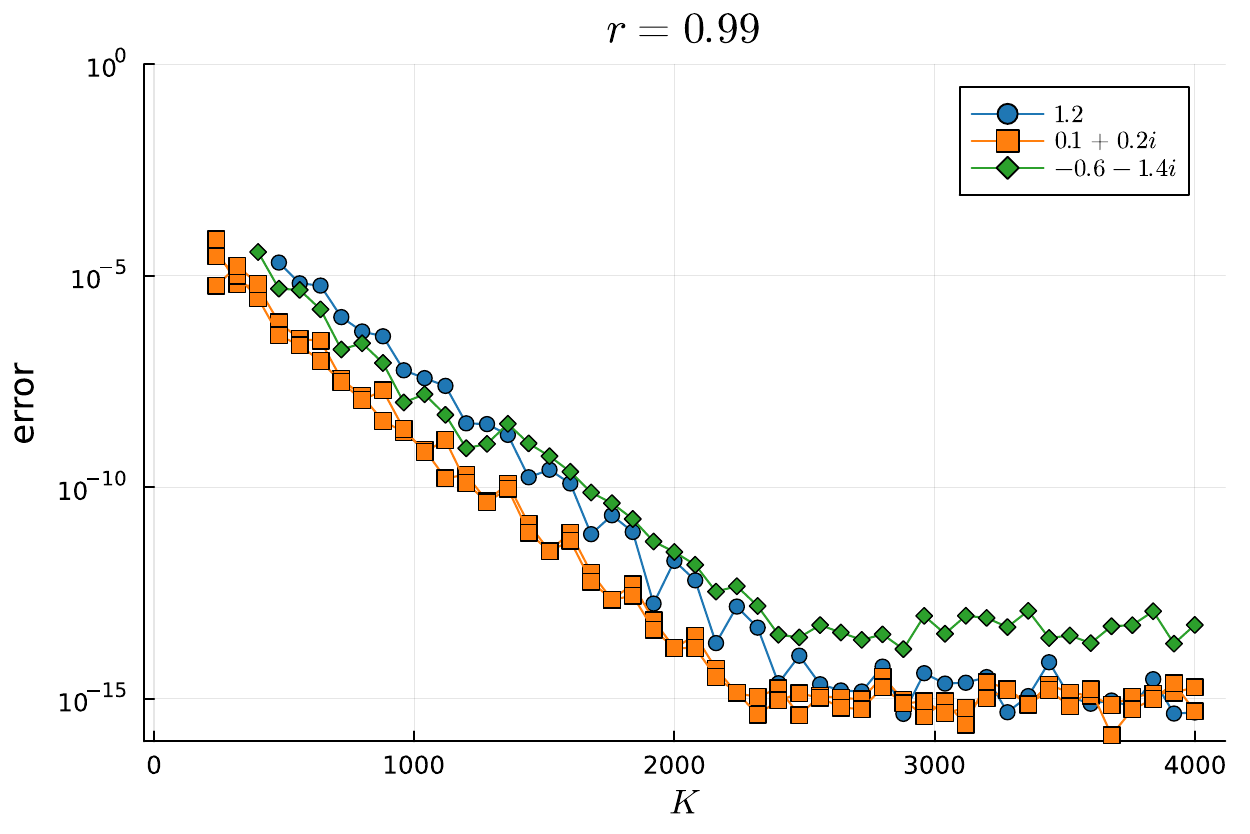}}
    \subfloat[]{\includegraphics[width = 0.5\textwidth]{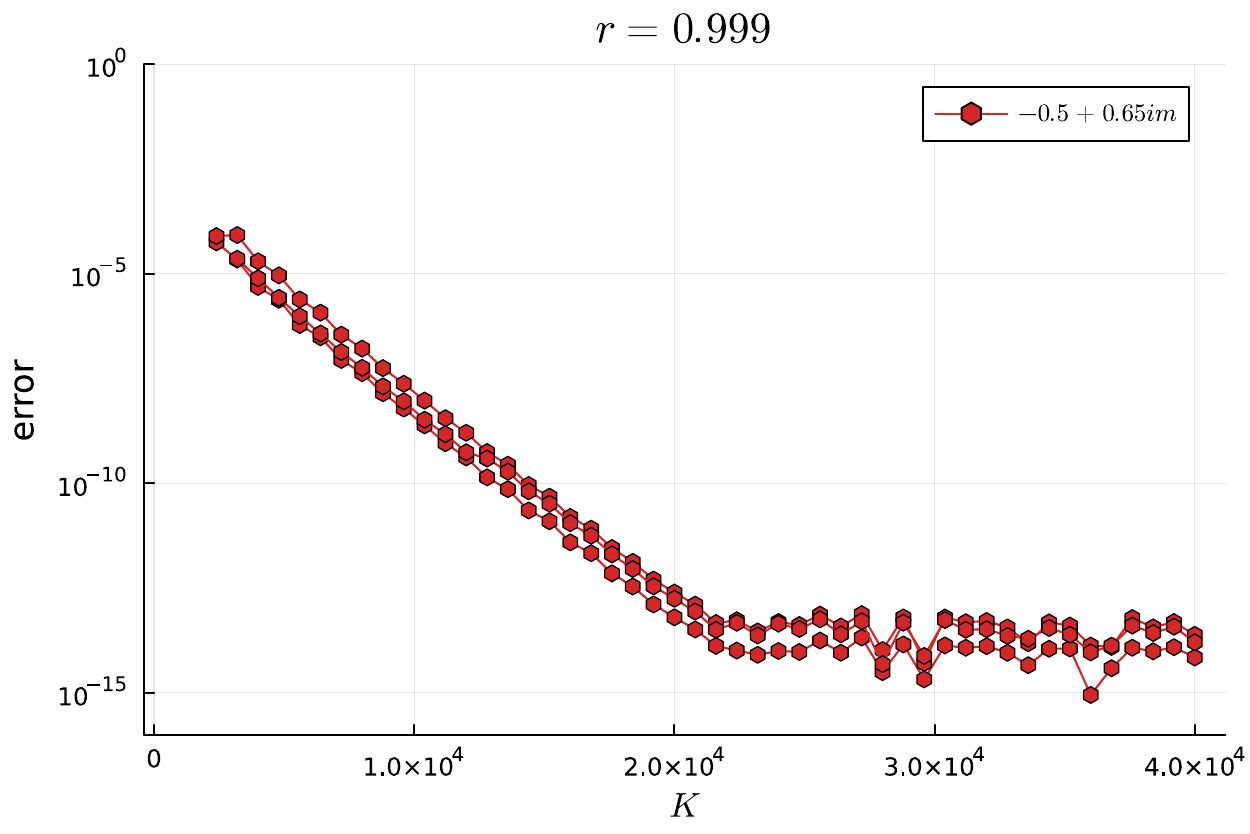}}
    \caption{} 
    \label{fig:error2}
\end{figure}

\begin{remark}
    The high number of quadrature points is not a problem in practice, since we can precompute the evaluations of $G_{\mu}$ at the quadrature points in advance. These values of $G_\mu$ can be used to find solutions to $G_{\mu}(z) = \zeta$ for any $\zeta$, whose computation is reduced to a few dot products and one low-dimensional eigenvalue calculation.
\end{remark}

\subsection{Free additive convolution of Marchenko--Pastur distributions}
\label{subsection:numericalexperiments3}
We will apply the algorithm developed in~\cite{olver2013numerical} to compute the \textit{free additive convolution} of measures that may be supported on disconnected sets or may not be absolutely continuous. We consider the \textit{Marchenko--Pastur distributions}~\cite{marchenkopastur1967} which are parametrised with $c \in (0,\infty)$. 
The distributions are defined as
\begin{align*}
    \label{eq:mpdist}
    \mu_c = \left\{\begin{matrix}
        (1-c) \delta_0 + \nu_c & c < 1 \\
        \nu_c & c \geq 1
    \end{matrix}\right., \\
    \dd \nu_c = \frac{\sqrt{\left(b-x\right)\left(x-a\right)}}{2\pi x} \ind_{[a,b]}(x) \dd x,
\end{align*}
where  $a = (1-\sqrt{c})^2$ and $b=(1+\sqrt{c})^2$. Note that for $c < 1$, the measure is supported on a disconnected set, as well as not being absolutely continuous.
The Marchenko--Pastur distribution arises as the analogue of the Poisson distribution in free probability. Note that the definition of the Marchenko--Pastur distribution varies slightly, here we choose the distribution whose free cumulants are all equal to $c$, see \cite[Chapter 2, Remark 12]{mingospeicher2017free} for additional details. This version of the Marchenko--Pastur distribution has $R$-transform given by \cite{CIT-001}
\begin{equation}
    R(z) = \frac{c}{1-z}.
\end{equation}

We will numerically compute the \textit{free additive convolution}~\cite{VOICULESCU1986323} (denoted by $\boxplus$) of two Marchenko--Pastur distributions, which is already known in closed form by the following relation, which follows from the formula for the $R$-transform:
\begin{equation*}
    \mu_{c_1} \boxplus \mu_{c_2} = \mu_{c_1 + c_2}.
\end{equation*}

We choose to numerically compute the convolution with $c_1 = \frac 12$ and $c_2 = 2$ (Figure~\ref{fig:density4}) since we then know the output must be a measure which is square-root decaying at the boundary (a necessary condition for recovery in the method applied).

\begin{figure}[H]
    \centering
    \subfloat[]{\includegraphics[width = 0.33\textwidth]{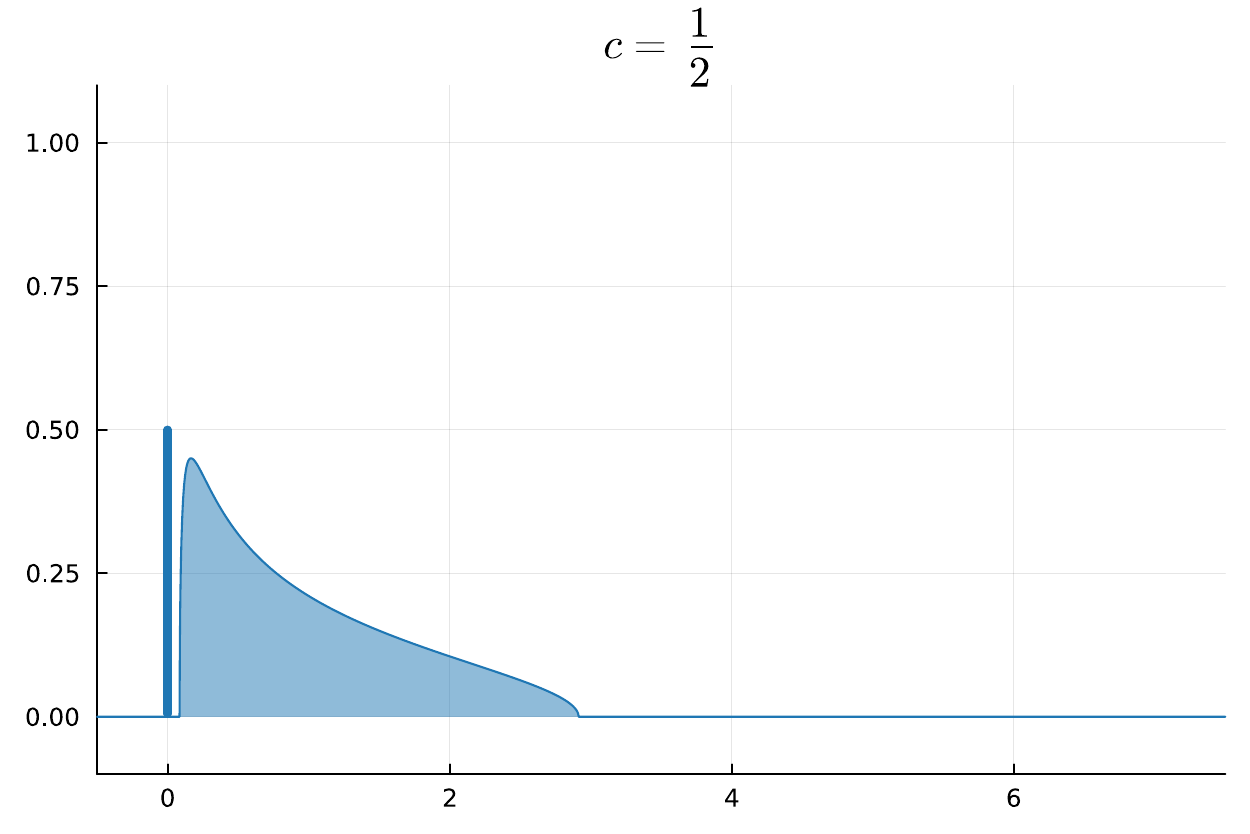}}
    \subfloat[]{\includegraphics[width = 0.33\textwidth]{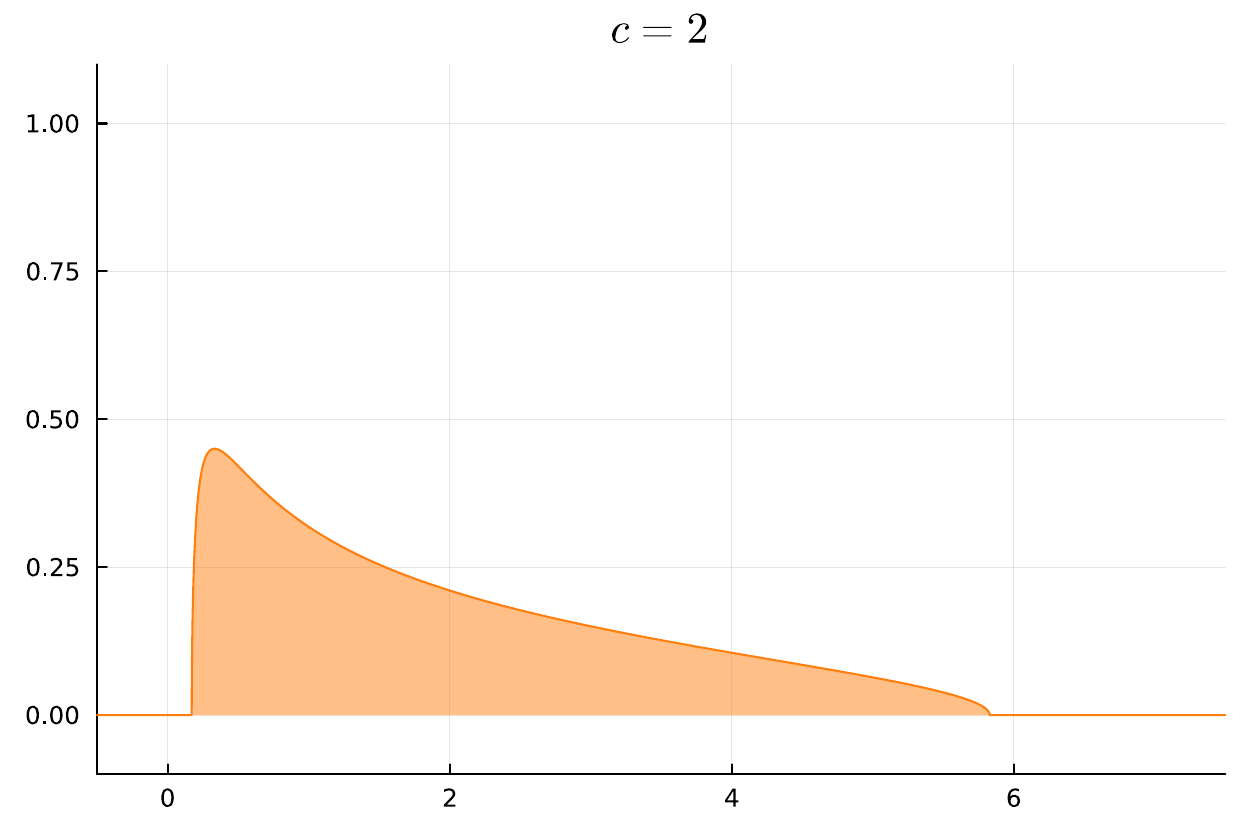}}
    \subfloat[]{\includegraphics[width = 0.33\textwidth]{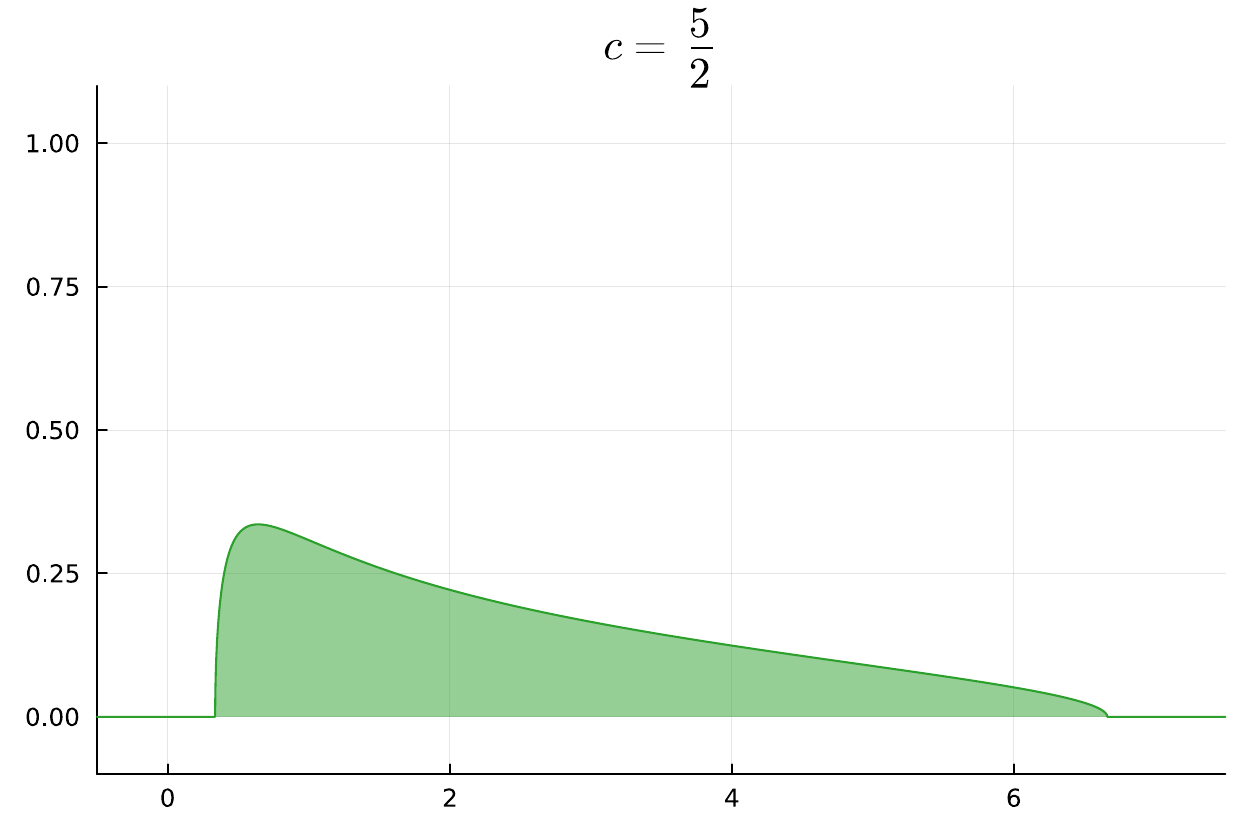}}
    \caption{Marchenko--Pastur distributions for various values of $c$. Pure point components of measures are denoted with a thick line. The free convolution of the left and middle measures is the measure on the right.}
    \label{fig:density4}
\end{figure}

We compute the convolution numerically using $1000$ quadrature points and $r=0.95$ for the calculations of $G_{\mu_{c_i}}^{-1}$. A bound for the number of inverses can be obtained by approximating the pure point with absolutely continuous Lipschitz spikes as in Example~\ref{example:purepoint} and apply Proposition~\ref{proposition:weakuniform} to see that for $c=\frac 12$ there are at most $2$ inverses whilst for $c=2$ there is at most $1$. In fact, the Marchenko--Pastur distributions have univalent Stieltjes transform for all values of $c$, though the theorems presented in this paper are only able to prove univalence for the case $c \geq 1$. The left inverse can be calculated explicitly and is given by 
\begin{equation*}
    G_{\mu_c}^{-1}(z) = \frac 1z + \frac{c}{1-z}
\end{equation*}
This follows from the relation $R(z) = G^{-1}(z) - \frac{1}{z}$.

For the purposes of computation, we expand each of the measures in terms of Chebyshev polynomials of the second kind (Algorithm 4 in \cite{olver2013numerical}). More precisely, we have an expansion of the form
\begin{equation}
    \label{eq:expansionofmeasure}
    \dd \mu = \frac{2\sqrt{(b-x)(a-x)}}{b-a} \sum_{k=0}^{\infty}\phi_k U_k(M_{(a,b)}^{-1}(x)) \ind_{[a,b]}(x)\dd x
\end{equation}
where the $\phi_k$ are the coefficients we wish to compute and the polynomials $U_k$ satisfy the recursion relation
\begin{align*}
    U_0(x) = 1 \hspace{1cm} U_1(x) = 2x \\
    xU_k(x) = \frac 12 U_{k-1}(x) + \frac 12 U_{k+1}(x)
\end{align*}
In this case, the coefficients $\phi_k$ of the expansion in \eqref{eq:expansionofmeasure} for the density of the Marchenko--Pastur distribution $\mu_{c}$ for $c>1$ can be computed exactly, given by
\begin{equation}
    \label{eq:mpcoefficients}
    -\frac{(-\sqrt{c})^{-(k+1)}}{\pi}.
\end{equation}
This can be proven by using a change of variable to first show that 
\begin{equation*}
    \phi_k = \frac{1}{\pi^2}\int_{-1}^{1}\sqrt{1-x^{2}}U_k(x)\left(\frac{1}{x+\frac{1+c}{2\sqrt{c}}}\right)\dd x \hspace{1cm}
\end{equation*}
The result coefficients can then be evaluated by direct computation for $k=0,1$ and then by utilising the three-term recurrence relation to show that for $k \geq 2$
\begin{equation*}
    \phi_k = -\frac{1+c}{\sqrt{c}}\phi_{k-1}-\phi_{k-2}
\end{equation*}
The formula in \eqref{eq:mpcoefficients} thus follows by induction.

The method for recovering the output probability measure involves solving a Vandermonde system for the first $n$ coefficients of the series $\phi_k$. We compute these coefficients for various values of $n$ and demonstrate convergence as $n$ increases.

The error is measured in two different ways. Since the coefficients are known in closed form, we plot the absolute difference of the computed and actual coefficients for $n=0$ to $19$ (Figure~\ref{fig:error4}, Left). We also measure the pointwise convergence of the densities at $100$ Chebyshev nodes $x_i$ by taking the maximum of the absolute errors $|\rho_n(x_i) - \rho(x_i)|$, where $\rho$ is the actual density and $\rho_n$ is the computed density (Figure~\ref{fig:error4}, Right). We note that the error plateauxs at around $n=75$. This is because the remaining Chebyshev coefficients are on the order of machine precision (in the case of the Marchenko-Pastur distribution with $c=\frac 52$, we have that $\phi_{74} \approx 3 \times 10^{-16}$).

\begin{figure}[H]
    \centering
    \subfloat[]{\includegraphics[width = 0.5\textwidth]{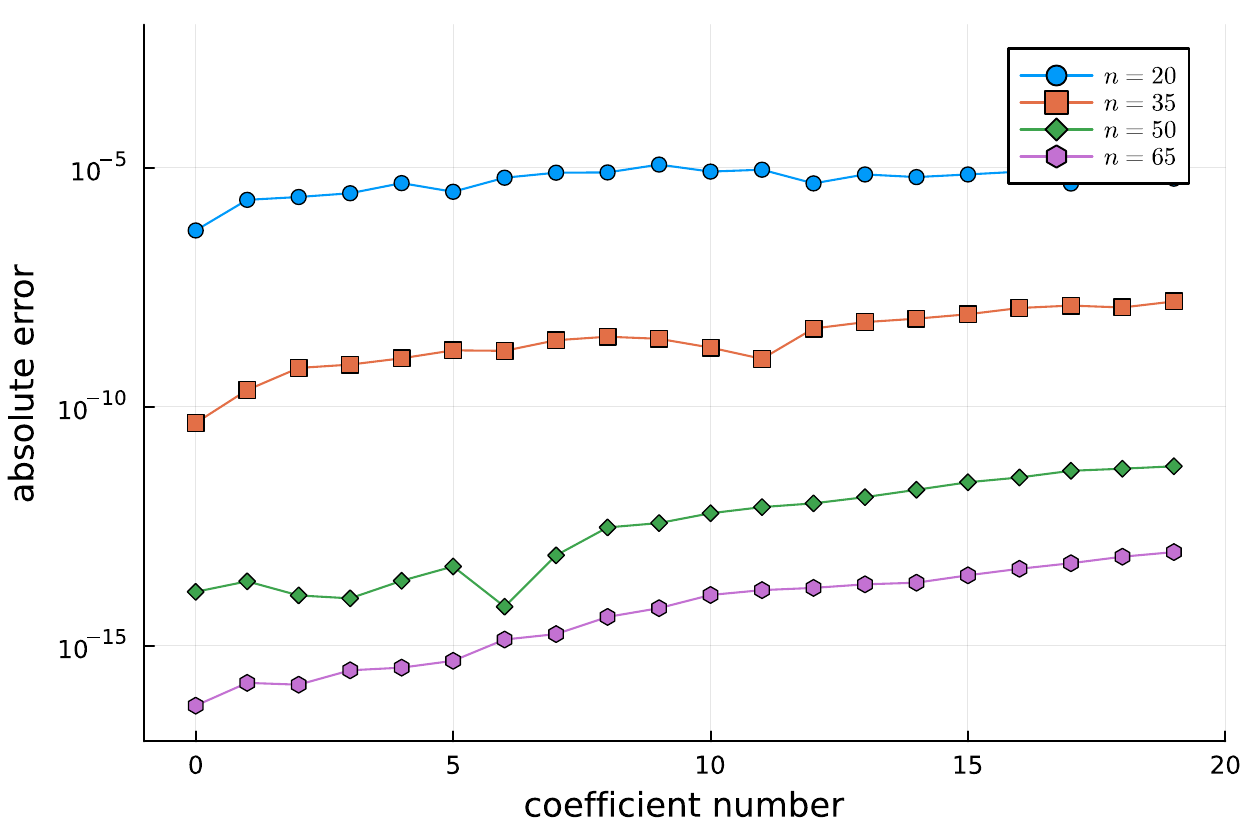}}
    \subfloat[]{\includegraphics[width = 0.5\textwidth]{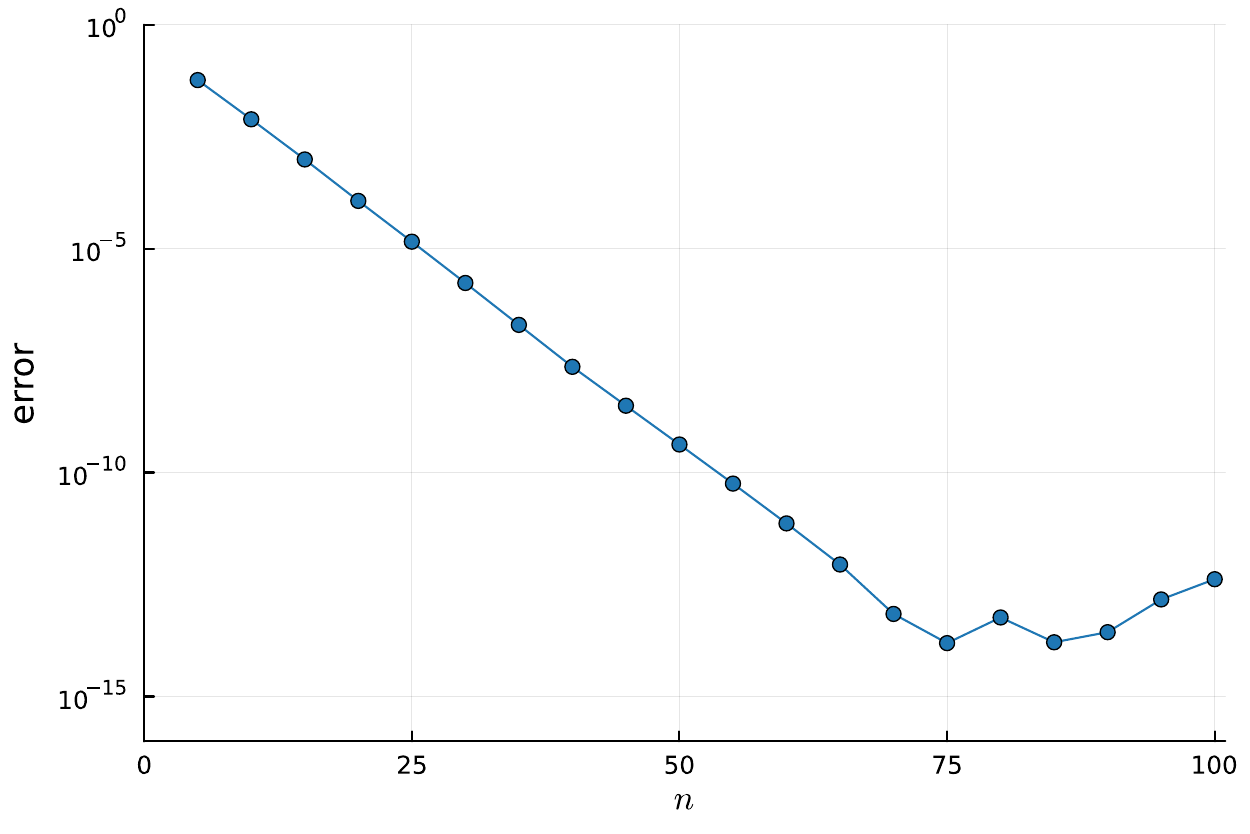}}
    \caption{Left: Absolute difference between $\phi_k$, the actual coefficients and $\phi_{k,n}$, the computed coefficients for $k=0$ to $19$. Right: The maximum absolute pointwise error over $100$ Chebyshev nodes, plotted against the number of recovered coefficients $n$. }
    \label{fig:error4}
\end{figure}

\subsection{Free convolution of Jacobi measures}
As a final example, we will consider two Jacobi weight measures with unbounded behaviour in their densities and compute their free convolution with the method in Section \ref{section:computing}. Let the measures $\mu_1$ and $\mu_2$ admit densities $\rho_1$ and $\rho_2$, given by

\begin{align*}
    \rho_1(x) &= \frac{1}{Z_1}(2-\sin(x)) (x+1)^{\frac 12} (1-x)^{-\frac 12}, \\
    \rho_2(x) &= \frac{1}{Z_2}\left(3+\exp\left(-\frac x4\right)\right) (x+1)^{-\frac 23} (1-x)^{\frac 13},
\end{align*}
where $Z_1$ and $Z_2$ are normalisation constants, see Figure \ref{fig:density5}.
\begin{figure}[H]
    \centering
    \subfloat[]{\includegraphics[width = 0.5\textwidth]{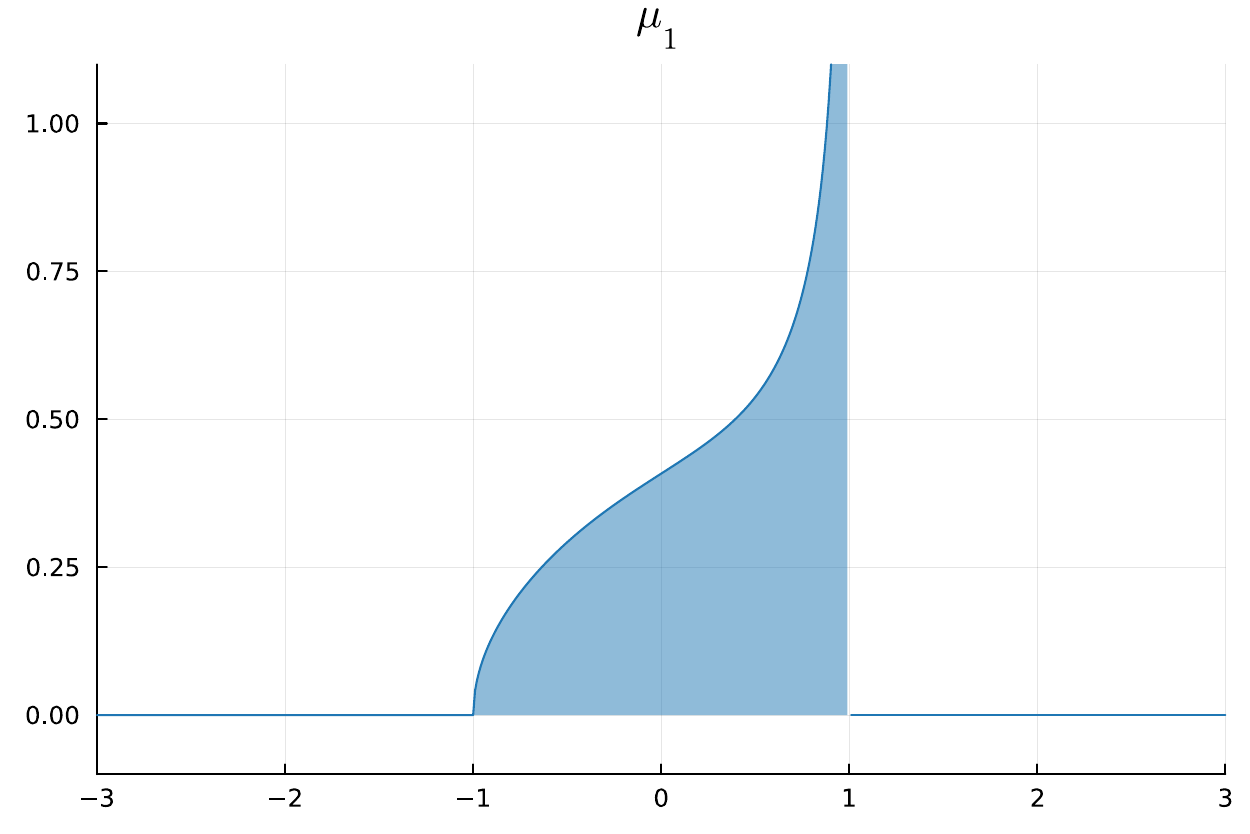}}
    \subfloat[]{\includegraphics[width = 0.5\textwidth]{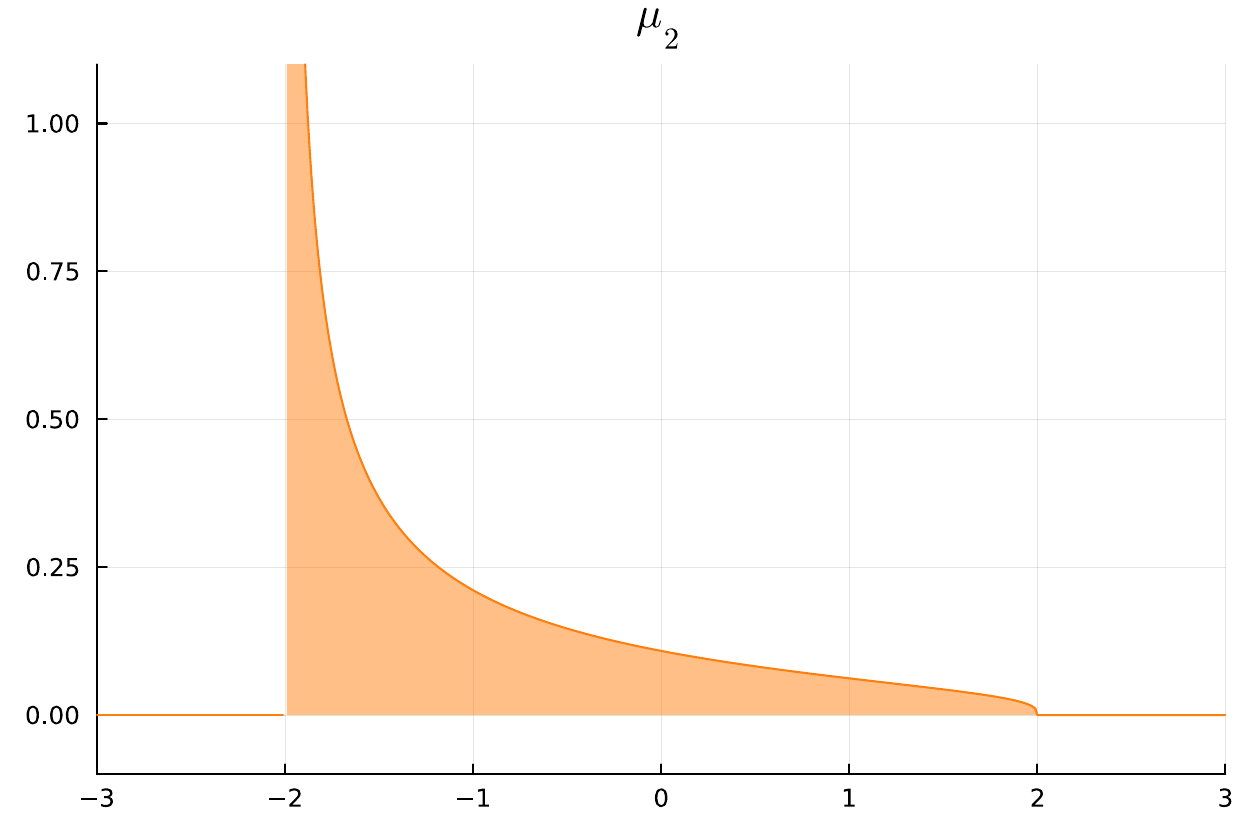}}
    \caption{Densities of the measures $\mu_1$ and $\mu_2$.}
    \label{fig:density5}
\end{figure}

Both measures can be shown to be univalent by Theorem \ref{theorem:bounds3}. It can then be shown by the existence of subordination functions  \cite[Theorem 4.1]{Belinschi2007} that the Stieltjes transform of the output measure is univalent as well, and that 
\begin{equation*}
    G^{-1}_{\mu_1 \boxplus \mu_2}(\zeta) = G_{\mu_1}^{-1}(\zeta) + G_{\mu_2}^{-1}(\zeta) - \frac 1\zeta
\end{equation*}
holds for $\zeta \in \C^+$ if and only if $G_{\mu_1}^{-1}(\zeta) + G_{\mu_2}^{-1}(\zeta) - \frac 1\zeta$ lies in $\C^-$. The same result holds with the upper and lower halves of $\C$ switched by complex conjugation.

In \cite[Theorem 2.2]{Bao2020}, we know that the output of the convolution $\mu_1 \boxplus \mu_2$ is an absolutely continuous measure supported on a single interval with square-root behaviour at the boundary. Hence, we may use the recovery method for square-root measures detailed in \cite{olver2013numerical}.

\begin{figure}[H]
    \centering
    \includegraphics[width = 0.5\textwidth]{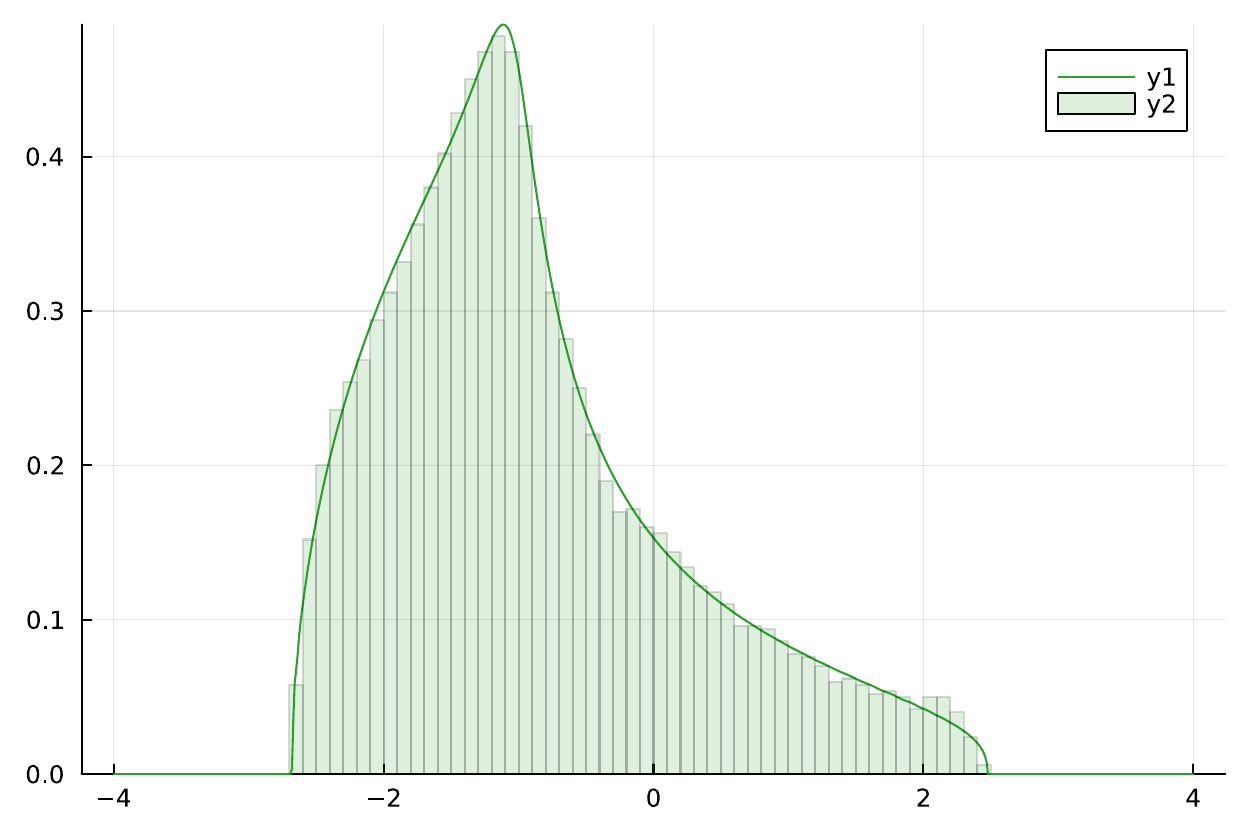}
    \caption{Density of the convolution of $\mu_1$ and $\mu_2$, with a histogram of the eigenvalues of a $5000 \times 5000$ matrix.}
    \label{fig:density5c}
\end{figure}

In Figure \ref{fig:density5c} we plot the computed density. Although no closed form for the density exists, we may approximate the density by considering the eigenvalues of the matrix $A_n + U_n B_n U_n^*$, where $A_n$ and $B_n$ are large $n \times n$ matrices with eigenvalue distributions of $\mu_1$ and $\mu_2$ respectively, and $U_n$ is a Haar distributed unitary matrix.

\section{Conclusion}
\label{section:conclusion}
We prove that the number of functional inverses of the Stieltjes transform of a probability measure is bounded above by easily observable properties of the measure. Using this, we present methods for rigorously finding all solutions to the inverse Stieltjes transform $G_{\mu}^{-1}$ for a broad class of measures, including measures with support on multiple intervals. The main application is towards the computation of free convolutions of probability measures. We apply these methods to compute free convolutions of measures with unbounded behaviour and disconnected support, provided the output is supported on a single interval with square-root behaviour at the boundary.

In \cite{olver2013numerical}, the multivalued properties of Stieltjes transforms was identified as a major complication of
the convolution of measures supported on multiple intervals.
This paper opens up the possibility of numerically computing free convolutions of measures with multiple intervals of support, which will be the subject of future work. Preliminary results have shown success in computing free convolutions of measures where both measures are supported on multiple intervals.

\appendix
\section{Special cases of computing inverse Stieltjes transforms}
\subsection{Pure point measures}

\label{appendix:pp}
Let $\mu$ be a pure point probability measure that is the sum of $N$ Dirac measures.
    \begin{equation*}
        \mu = \sum_{i=1}^N \alpha_i \delta_{x_i}
    \end{equation*}
    where $x_i$ are pairwise distinct and $\sum_{i} \alpha_i = 1$.
then the Stieltjes transform is a rational function. Computing the solutions of $G_{\mu}(z) = \zeta$ therefore reduces to finding the roots of a degree $N$ polynomial whose coefficients depend on $\zeta$.

\subsection{Square-root measures supported on one interval}
\label{appendix:sq}
Let $\mu$ be a measure (WLOG supported on $[-1,1]$) whose density can be expressed in the form $\rho(x) = r(x)w(x)$, where $r(x)$ is a bounded function with sufficient smoothness. Consider the Chebyshev polynomials of the 2nd kind $U_n(x)$ and their Stieltjes transforms $q_n:\C \setminus [-1,1] \to \C$ defined as
\begin{equation*}
    q_n(z) = \int_{-1}^{1}\frac{U_n(x)}{z-x}\sqrt{1-x^2}\dd x.
\end{equation*}
Then, it can be shown by direct computation that
\begin{equation*}
    q_n(z) = \pi(z - \sqrt{z^2-1})^{n+1} = \pi (J^{-1}_+(z))^{n+1}
\end{equation*}
where $J^{-1}_+$ is the inverse of the Joukowski transform mapping $\C \cup \set{\infty} \setminus [-1,1]$ to $\mathbb{D}$ the open unit disc~\cite{olver2013numerical}. Hence, by expanding the Stieltjes transform and truncating the series as in \eqref{eq:truncatedcauchytransform}, finding the solutions of $G_{\mu}(z) = \zeta$ is equivalent to finding the roots of a polynomial which lie inside the open unit disc.

\printbibliography

\end{document}